\documentclass[11pt]{amsart}

\usepackage{amsthm}
\usepackage{amssymb}
\usepackage{latexsym}
\usepackage{euscript}
\usepackage{amsmath} 
\usepackage{amstext}
\usepackage{amsgen}
\usepackage{amsbsy}
\usepackage{amsopn}
\usepackage{amsfonts}
\usepackage{hyperref}

\usepackage{thmtools} 
\usepackage[capitalize]{cleveref}
\crefname{theorem}{Theorem}{Theorems}
\Crefname{theorem}{Theorem}{Theorems}
\crefname{lemma}{Lemma}{Lemmas}
\Crefname{lemma}{Lemma}{Lemmas}
\crefname{proposition}{Proposition}{Propositions}
\Crefname{proposition}{Proposition}{Propositions}
\crefname{corollary}{Corollary}{Corollaries}
\Crefname{corollary}{Corollary}{Corollaries}
\crefname{definition}{Definition}{Definitions}
\Crefname{definition}{Definition}{Definitions}
\crefname{example}{Example}{Examples}
\Crefname{example}{Example}{Examples}
\crefname{remark}{Remark}{Remarks}
\Crefname{remark}{Remark}{Remarks}
\crefname{equation}{}{}
\Crefname{equation}{}{}

\usepackage{graphicx}
\usepackage{bbm}
\usepackage{tikz} 
\usepackage{tikz-cd}

\usepackage{enumerate}
\usepackage{color}
\usepackage{a4wide}
\usepackage{mathrsfs}
\usepackage{multirow}
\usepackage{mathtools}

\hypersetup{
    colorlinks=true,
    linkcolor=red}

\usetikzlibrary{matrix,arrows,decorations.pathmorphing}

\theoremstyle{definition}

\newtheorem{theorem}{Theorem}[section]
\newtheorem{proposition}[theorem]{Proposition}
\newtheorem{corollary}[theorem]{Corollary}
\newtheorem{lemma}[theorem]{Lemma}

\newtheorem{thmx}{Theorem}

\theoremstyle{definition}

\newtheorem{example}[theorem]{Example}

\newtheorem{remark}[theorem]{Remark}

\newcommand{\defn}[1]{\emph{#1}}

\newcommand{\ie}{i.e.\ }
\newcommand{\eg}{e.g.\ }

\newcommand{\cf}{cf.\ }

\newcommand{\N}{\mathbb{N}}
\newcommand{\Z}{\mathbb{Z}}

\newcommand{\R}{\mathbb{R}}
\newcommand{\C}{\mathbb{C}}

\newcommand{\kk}{\mathbbm{k}}
\newcommand{\PP}{\mathbb{P}}

\newcommand{\ob}[1]{\mathcal{#1}}

\newcommand{\cat}[1]{\mathsf{#1}}
\renewcommand{\mod}[1]{{#1}\text{-}\cat{mod}}
\newcommand{\orth}{ {}^\perp}

\newcommand{\Filt}[1]{\mathrm{filt}(#1)}

\newcommand{\depth}[1]{\,\mathrm{depth(#1)}}
\newcommand{\loc}[1]{\mathrm{Loc}(#1)}
\newcommand{\shf}[1]{\mathrm{Sh}( #1)}
\newcommand{\perv}[2]{{}^{#1}\cat{Perv}(#2)}

\newcommand{\pfun}[2]{^{#1}{#2}}

\newcommand{\palg}[2]{^{#1\hspace{-3pt}}A(#2)}

\newcommand{\constr}[2]{\cat{D}_{#1}({#2})}
\newcommand{\der}[2]{\cat{D}^{#1}({#2})}

\newcommand{\triang}[2]{\mathrm{tr}_{#1}(#2)}
\newcommand{\thick}[2]{\mathrm{th}_{#1}(#2)}

\newcommand{\id}{\mathrm{id}}

\newcommand{\mor}[3]{{\mathrm{Hom}^{#1}}\!\left(#2,#3\right)}
\newcommand{\Mor}[3]{{\mathrm{Hom}}_{#1}\!\left(#2,#3\right)}

\newcommand{\Ext}[4]{{\mathrm{Ext}_{#1}^{#2}}\!\left(#3,#4\right)}
\newcommand{\End}[1]{\mathrm{End}\!\left(#1\right)}
\newcommand{\EndA}[2]{\mathrm{End}_{#1}\!\left(#2\right)}

\newcommand{\sh}[1]{\mathcal{#1}}

\newcommand{\pt}{{\mathsf{pt}}}
\newcommand{\codim}{\mathrm{codim}\, }
\newcommand{\proje}[1]{\mathrm{proj}(#1)}
\newcommand{\strata}[1]{\mathscr{S}(#1)}
\newcommand{\pdim}[1]{\delta(#1)} 

\newcommand{\pdl}[2]{ {^{#1}}\cat{D}^{\leq{#2}}}
\newcommand{\pdg}[2]{ {^{#1}}\cat{D}^{\geq{#2}}}
\newcommand{\pbar}{\overline{p}}
\newcommand{\dups}{p^*}
\newcommand{\gdim}[1]{\mathrm{gldim}\left(#1\right)} 

\newcommand{\vect}[1]{\mathrm{Vect}_{#1}}

\newcommand{\monic}{\hookrightarrow}

\newcommand{\stan}[2]{\pfun{#1}{\Delta_{#2}}}
\newcommand{\costan}[2]{\pfun{#1}{\nabla_{\! #2}}}

\renewcommand{\star}[1]{\mathrm{st}(#1)}

\newcommand{\rk}[1]{\,\mathrm{rk}(#1)}

\makeatletter
\@namedef{subjclassname@2020}{\textup{2020} Mathematics Subject Classification}
\makeatother

\title{Faithful perversities}
\author{Alessio Cipriani and Jon Woolf}
\address{Alessio Cipriani, Dipartimento di Informatica - Settore di Matematica, Universit\`a degli Studi di Verona, Strada le Grazie 15 - Ca' Vignal, I-37134 Verona, Italy}
\email{alessio.cipriani@univr.it}
\address{Jon Woolf, Dept. of Mathematical Sciences, University of Liverpool, L69 7ZL, U.K.}
\email{jonwoolf@liverpool.ac.uk}

\begin{document}

\begin{abstract}
We show that the faithful highest weight hearts in an algebraic triangulated category are the serially faithful glued hearts, equivalently the hearts containing a dual pair of full exceptional collections in the sense of Bodzenta--Bondal \cite{Bodzenta:2026aa}. We then characterise faithful highest weight categories of perverse sheaves on topologically stratified spaces algebraically, in terms of the exactness of certain functors, and topologically, in terms of the vanishing of certain cohomology groups of pairwise links. 

We prove that the global dimension of a faithful category of perverse sheaves on a topologically stratified space $X$ with finitely many strata is bounded by the dimension of $X$. Finally, we show that in this setting the hypercohomology of a perverse sheaf can be computed from a projective resolution of the constant sheaf, and conversely that the multiplicities of the terms in a minimal projective resolution of the constant sheaf can be computed as intersection cohomology groups.
\end{abstract}


\maketitle

\section{Introduction}

Categories of perverse sheaves on a topologically stratified space were introduced in \cite{BBD} as hearts of bounded t-structures on the constructible derived category $\constr{c}{X}$ of a stratified space $X$. Perverse sheaves play an important role in studying the topology of singular spaces and have applications in stratified Morse and Hodge theory, geometric representation theory, Kazhdan--Lusztig theory, and $\mathscr{D}$-modules.

The category $\perv{p}{X}$  of perverse sheaves on a topologically stratified space $X$ depends on a perversity $p$, that is on an assignment of an integer to each stratum of $X$. Among the possible perversities one can consider, the ones which depend on the dimension and grow in a controlled way, which we call GM-perversities --- see \cref{subsec:perversities} --- are the most important. For these, the category of perverse sheaves captures interesting topological information about the space $X$: for instance, simple perverse sheaves are the intersection cohomology complexes in the sense of \cite{GMIH1}. 

When $X$ has finitely many strata, each with finite fundamental group, the category $\perv{p}{X}$ is equivalent to the category $\mod{\palg{p}{X}}$ of finite-dimensional modules over a finite dimensional algebra $\palg{p}{X}$. In this paper, we study the special situation in which $\perv{p}{X}$ is a faithful heart, \ie in which the equivalence $\perv{p}{X} \simeq \mod{\palg{p}{X}}$ extends to a triangulated equivalence $\constr{c}{X} \simeq \der{b}{\mod{\palg{p}{X}}}$. This is interesting because it is the case in which the topological information in the constructible derived category is accessible from the algebra, and {\it vice versa}. As a consequence, algebraic properties of the category $\perv{p}{X}$ can be expressed in terms of topological properties of the space $X$, see for instance \cref{thmB} below. There are two main classes of examples of faithful perverse hearts coming from simplicial topology and complex geometry respectively:
\begin{enumerate} 
    \item Polishchuk \cite[Theorem 4.2]{MR1453053} shows that $\perv{p}{X}$ is faithful when $X$ is a compact space stratified by a triangulation, and $p$ is any GM-perversity;
    \item Beilinson, Ginzburg and Soergel \cite[Corollary 3.3.2]{MR1322847} show $\perv{p}{X}$ is faithful for the middle perversity $p$ when $X$  is a smooth complex projective variety with affine strata, and Beilinson, Bezrukavnikov and Mirkovic \cite[\S1.5]{MR2119139} extends this to the situation in which $X$  is a complex variety stratified by algebraic subvarieties so that each stratum is smooth and affine with vanishing higher cohomology.
\end{enumerate} 

In Section~\ref{subsec:FH} we introduce the notion of \defn{serially faithful hearts} with respect to an admissible thick filtration of the ambient triangulated category as the iterated version of faithful hearts.  We show that serially faithful hearts which are glued along a full exceptional collection --- see Section~\ref{sec:thick subcats} --- are precisely the faithful highest weight categories. The following characterisation of faithful highest weight hearts follows from the results of \cite{Bodzenta:2026aa}.
\begin{thmx}[\cref{thm:fhw}]\label{thmA}
Let $\cat{D}$ be an algebraic triangulated category. The following are equivalent:
\begin{enumerate}
\item The heart $\cat{H}\subset \cat{D}$ is faithful and highest weight;
\item The sequence $(\stan{}{0},\ldots,\stan{}{n})$ of standard objects is a full exceptional collection with left dual pair the sequence $(\costan{}{n},\ldots,\costan{}{0})$ of  costandard objects;
\item The heart $\cat{H}$ contains a dual pair of full exceptional collections;
\item The heart $\cat{H}$ is  glued along a full exceptional collection and serially faithful for the associated admissible thick filtration.
\end{enumerate} 
\end{thmx}

Examples of serially faithful perverse hearts (with respect to the admissible thick filtration associated to any filtration  by closed unions of strata) are given by the faithful hearts studied in \cite{MR2119139,MR1322847} and \cite{MR1453053}. We provide an example of a faithful glued heart inside the constructible derived category which is not highest weight and therefore not serially faithful (for any ordering of the simple objects). However, this heart is not a category of perverse sheaves for any perversity and we do not know an instance of a faithful perverse heart which is not serially faithful. 

When $\cat{H}$ is a category of perverse sheaves inside the constructible derived category of a topologically stratified space the conditions of \cref{thmA} can be interpreted topologically to  characterise faithful perverse highest weight hearts in terms of the exactness of certain functors, or in terms of the vanishing of certain cohomology groups of pairwise links.

\begin{thmx}[\cref{thm:fhwp}]
\label{thmB}
Let $X$ be a topologically stratified space with finitely many contractible strata and denote by $\imath_S\colon S\to X$ the inclusion of a stratum into $X$.  The following are equivalent:
\begin{enumerate}
\item the category $\perv{p}{X}$ is faithful highest weight for some (hence any) total order refining the partial order of strata;
\item for any stratum $S\subset X$ the functors ${\imath_S}_*$ and ${\imath_S}_!$ are t-exact;
\item for any pair of distinct strata $S\subset\overline{T}$ the pairwise link $L_{S<T}$ satisfies 
\[
H^{>p(S)-p(T)}(L_{S<T})=0\ \text{and}\ H^{<p(S)-p(T)-1}_c(L_{S<T})=0.
\]

\end{enumerate}
\end{thmx}

The equivalence of $(2)$ and $(3)$ (expressed as the vanishing of certain cohomology sheaves) and the fact that they imply $(1)$ can be found in \cite{ParshallScott}, with a different proof. The other implication is new so far as we are aware. We note that if $X$ has finitely many 1-connected strata with second vanishing cohomology and either the second or the third condition of \cref{thmB} are satisfied, then $\perv{p}{X}$ is highest weight but not necessarily faithful, see \cite{ParshallScott}. The converse of this result does not hold; see \cref{constr_sheaves_P1} for a counterexample. This reflects the fact that the faithfulness assumption allows one to interpret topological information on perverse sheaves algebraically, and vice-versa.

We further exploit the tight connection between algebra and topology for faithful perverse hearts in order to obtain cohomological vanishing results. These can be used to bound the global dimension of categories of perverse sheaves by the dimension of the space, and in the case of faithful highest weight perverse hearts also in terms of the poset of strata. 

\begin{thmx}[\cref{on gl dim} and \cref{cor: gldim depth bound}]
\label{thmC}
Let X be a topologically stratified space with finitely many strata and p a GM-perversity. For any $\sh{E},\sh{F}\in\perv{p}{X}$ we have
\begin{equation*}
\Ext{\constr{c}{X}}{k}{\sh{E}}{\sh{F}}=0\ \ \mbox{if}\ k<0 \ \text{or}\  k>\dim(X).  
\end{equation*}
Moreover,  $\gdim{\perv{p}{X}}\leq\dim(X)$ when $\perv{p}{X}$ is faithful, and $\gdim{\perv{p}{X}}\leq 2 \depth{X}$ when it is faithful highest weight.
\end{thmx}

The difference between the two bounds reflects the different proof strategies: the first uses an inductive argument and vanishing results for the hypercohomology of perverse sheaves, while the second uses \cref{thmB}. In Section~\ref{satisfaction_gd_bounds} we investigate lower bounds for the global dimension and conditions under which the above bounds  are achieved. We note that in the setting of \cref{thmC} it follows that the constructible derived category $\constr{c}{X}$ admits a Serre functor which, in the special case of the complex projective space with affine stratification and the middle perversity can be described as a $\mathbb{P}$-twist \cite{Huybrechts_2006} at the simple perverse sheaf corresponding to the top dimensional stratum, see \cite{Bonfert:2025aa}.

In Section~\ref{hypercoho_computations} we make use of the faithfulness assumption on $\perv{p}{X}$ to compute topological invariants algebraically, and conversely to compute algebraic data topologically.  When $X$ is a topologically stratified space with finitely many strata, each with finite fundamental group, $\perv{p}{X}$ has enough injective and projective objects by \cite{MR4349392}. Moreover, if $\perv{p}{X}$ is faithful in $\constr{c}{X}$, the global dimension of the perverse heart is finite by \cref{thmC}. Therefore projective resolutions of objects in $\perv{p}{X}$ exist and are finite. We show that the hypercohomology of a perverse sheaf can be computed from a projective resolution of the constant sheaf, and conversely that the multiplicities of the terms appearing in a minimal projective resolution can be computed as intersection cohomology groups.

\begin{thmx}[\cref{prop:hypercoh} and \cref{cor:computation_cohom}]
\label{thmD}
Suppose $X$ is a topologically stratified space with finitely many strata, each with finite fundamental group. Let $p$ be a perversity on $X$ such that $\perv{p}{X}$  is a faithful heart in $\constr{c}{X}$. Then:
\begin{enumerate}
\item The complex of morphisms from a projective resolution of the constant sheaf to a perverse sheaf $\sh{E}$ computes the hypercohomology of $\sh{E}$.
\item When a shift of the constant sheaf is perverse the multiplicities of  indecomposable projectives in its {\em minimal} projective resolution are intersection Betti numbers. 
\end{enumerate}
There are dual results describing the hypercohomology in terms of an injective resolution of the constant sheaf, and the multiplicity of each injective hull in a minimal injective resolution when the constant sheaf is shifted perverse. 
\end{thmx}

We apply the results of Section~\ref{hypercoho_computations} when $X$ is the geometric realisation of a finite simplicial complex $K$. In \cref{proj_res_simplicial} we show the constant sheaf on $X$ admits a finite projective resolution $\sh{P}^*$ indexed by increasing perverse dimension as introduced by MacPherson, see \cite[\S23.6]{goresky}. \cref{cor:cellular perverse sheaves} shows that $\mor{}{\sh{P}^*}{\sh{E}}$ gives the quiver representation of a perverse sheaf $\sh{E}$, equivalently the cellular perverse sheaf description of \cite{goresky}. 

Finally, we point out another consequence of the presence of a faithful perverse heart that we do not further explore here. When $X$ has finitely many strata, each with finite fundamental group, and there is a faithful perverse heart $\perv{p}{X}$, then all the other length hearts in $\constr{c}{X}$ correspond to silting objects in the bounded homotopy category of projective perverse sheaves by the K\"onig--Yang correspondence \cite{MR3178243}. It would be interesting to understand the special features of the silting objects corresponding to perverse hearts, and faithful and serially faithful hearts.

\subsection*{Acknowledgments} The first author was funded by MUR PNRR-Seal of Excellence, CUP B37G22000800006, by the project funded by NextGenerationEU under NRRP, Call PRIN 2022 No. 104 of February 2, 2022 of Italian Ministry of University and Research; Project 2022S97PMY Structures for Quivers, Algebras and Representations (SQUARE), and by the project LAVIE --- Large views of small phenomena: decompositions, localizations and representation type --- Fondo Italiano per la Scienza FIS-2021 funded by of Italian Ministry of University and Research. A.C is member of the ``National Group for Algebraic and Geometric Structures, and their Applications" (GNSAGA - INdAM). Thanks to Lukas Bonfert for helpful comments.

\section{Categorical Preliminaries}
\label{sec:prelims} 

This section surveys the necessary background material on triangulated categories and hearts, introducing our notation and conventions. The material is almost all well-known, with the exception of the notion of serially faithful heart.

\subsection{Conventions and notation}

We work over a field $\kk$. Let $\cat{D}$ be a Hom-finite $\kk$-linear triangulated category with shift functor $[1]$. We set $\Ext{\cat{D}}{m}{\ob{D}}{\ob{D}'} \coloneqq \Mor{\cat{D}}{\ob{D}}{\ob{D'}[m]}$. 

By subcategory we mean strictly full subcategory, \ie full and closed under isomorphisms. A subcategory of $\cat{D}$ is \defn{triangulated} if it is closed under shift and extensions, and \defn{thick} if in addition it is closed under taking direct summands. For a subcategory $\cat{C}\subset \cat{D}$ we denote  the smallest triangulated and thick subcategories of $\cat{D}$ containing $\cat{C}$ respectively by $\triang{}{\cat{C}}$ and $\thick{}{\cat{C}}$. We let $\cat{C}\orth$ and $\orth\cat{C}$  be the subcategories on objects $\ob{D}\in \cat{D}$ with respectively $\Mor{\cat{D}}{\ob{C}}{\ob{D}}=0$ and $\Mor{\cat{D}}{\ob{D}}{\ob{C}}=0$ for all $\ob{C}\in \cat{C}$. 

For subcategories $\cat{B},\cat{C} \subset \cat{D}$ we let $\cat{B}*\cat{C}$ be the subcategory of all $\ob{D}\in \cat{D}$ sitting in a triangle $\ob{B} \to \ob{D}\to \ob{C} \to \ob{B}[1]$ with $\ob{B}\in \cat{B}$ and $\ob{C}\in \cat{C}$. A \defn{semi-orthogonal decomposition} $\langle \cat{B},\cat{C}\rangle$ of $\cat{D}$ is a pair of thick subcategories such that $\cat{B} = \orth\cat{C}$ and $\cat{D}=\cat{B}*\cat{C}$. 

\subsection{Hearts and t-structures}
\label{sec:hearts}

A \defn{t-structure} on the triangulated category $\cat{D}$ is a pair $(\cat{D}^{\leq 0},\cat{D}^{\geq 0})$ of additive subcategories such that the \defn{aisle} $\cat{D}^{\leq 0}$ is closed under $[1]$, the \defn{coaisle} $\cat{D}^{\geq 0} = (\cat{D}^{< 0})\orth$, and $\cat{D} = \cat{D}^{<0} *\cat{D}^{\geq 0}$. Here we use the notation $\cat{D}^{\leq n} \coloneqq \cat{D}^{\leq 0}[-n]$ and so forth. The inclusion $\cat{D}^{\leq 0} \hookrightarrow \cat{D}$ has right adjoint $\tau^{\leq 0}$, and $\cat{D}^{\geq 0} \hookrightarrow \cat{D}$ has left adjoint $\tau^{\geq 0}$.

The \defn{heart} of the t-structure is $\cat{H}\coloneqq \cat{D}^{\leq 0} \cap \cat{D}^{\geq 0}$. It is an abelian subcategory of $\cat{D}$. The t-structure is \defn{bounded} if $\cat{D} = \bigcup_{n\in \N} \cat{H}[n] * \cdots * \cat{H}[-n]$ in which case the heart determines the t-structure. Henceforth by heart we will always mean heart of a bounded t-structure.

\subsection{Thick subcategories and recollements}
\label{sec:thick subcats}

Let $\cat{D}$ be a triangulated category and $\cat{C} \subset \cat{D}$ a thick subcategory. If the subcategory is \emph{admissible}, \ie the inclusion $\imath_*$ has both left and right adjoints $\imath^*$ and $\imath^!$, then it determines a \emph{recollement}
\[
\begin{tikzcd} [row sep={0.6cm},column sep={0.6cm}]\label{equation:6ff}
\cat{C}\ar[rr,"\imath_*","\perp"']&&
\cat{D}\ar[ll,"\imath^!",bend left=35]\ar[ll,"\imath^*","\perp"', swap,bend right=50]\ar[rr,"\jmath^*","\perp"']&&
\cat{D}/\cat{C}.\ar[ll,"\jmath_*",bend left=35]\ar[ll,"\jmath_!","\perp"', swap,bend right=50]
\end{tikzcd}
\]
The left and right adjoints $\jmath_!$ and $\jmath^*$ to the quotient $\jmath^*\colon \cat{D} \to \cat{D}/\cat{C}$ are constructed using the fact that the restrictions $\jmath^*\colon \orth\cat{C} \to \cat{D}/\cat{C}$ and $\jmath^*\colon \cat{C}\orth \to \cat{D}/\cat{C}$ are equivalences \cite[Proposition 1.6]{MR1039961}. 
There are semi-orthogonal decompositions $ \langle \cat{C}, \cat{C}\orth\rangle =\cat{D} = \langle \orth\cat{C},\cat{C}\rangle$ arising from the adjunction triangles
\[
\imath_*\imath^! \ob{D} \to \ob{D} \to \jmath_*\jmath^*\ob{D} \to \imath_*\imath^!\ob{D}[1]
\quad \text{and} \quad
\jmath_!\jmath^!\ob{D} \to \ob{D} \to \imath_*\imath^*\ob{D} \to \jmath_!\jmath^!\ob{D}[1].
\]
Recall from \cite[Th\'eor\`eme 1.4.10]{BBD} that given t-structures on the admissible subcategory $\cat{C}$ and quotient $ \cat{D}/\cat{C}$ can be \defn{glued} to a t-structure on $\cat{D}$ whose  aisle is  the full subcategory
\[
\{ \ob{D} \in \cat{D} \mid \tau^{>0} \jmath^*\ob{D}=0 = \tau^{>0} \imath^*\ob{D} \}
\]
and whose co-aisle is $\{ \ob{D} \in \cat{D} \mid \tau^{<0} \jmath^*\ob{D}=0 = \tau^{<0} \imath^!\ob{D} \}$ where we abuse notation by using $\tau^{<0}$ and $\tau^{>0}$ to denote the truncation functors of the given t-structures in both $\cat{C}$ and $\cat{D}/\cat{C}$. The glued t-structure is bounded precisely when the given ones  are so.

The next result characterises glued hearts. We say that a semi-orthogonal decomposition $\cat{D} = \langle \cat{A} , \cat{B}\rangle$ restricts to a (not necessarily triangulated) subcategory $\cat{D}'$ if $\cat{D}' = (\cat{D}'\cap \cat{A}) * (\cat{D}'\cap\cat{B})$. 

\begin{proposition}
\label{prop: glued heart conditions}
The following are equivalent an admissible thick subcategory $\cat{C} \subset \cat{D}$:
\begin{enumerate}
\item The t-structure $(\cat{D}^{\leq 0},\cat{D}^{\geq 0})$ is glued from the induced recollement;
\item The functor $\jmath_!\jmath^*$ is right t-exact for the t-structure $(\cat{D}^{\leq 0},\cat{D}^{\geq 0})$;
\item The functor $\jmath_*\jmath^*$ is left t-exact  for the t-structure $(\cat{D}^{\leq 0},\cat{D}^{\geq 0})$;\item The semi-orthogonal decomposition $\cat{D}=\langle \orth\cat{C},\cat{C}\rangle$ restricts to the aisle $\cat{D}^{\leq 0}$;
\item The semi-orthogonal decomposition $\cat{D}=\langle \cat{C},\cat{C}\orth\rangle$ restricts to the co-aisle $\cat{D}^{\geq 0}$.
\end{enumerate}
\end{proposition}
\begin{proof}
Items (1), (2) and (3) are equivalent by \cite[Proposition 1.4.12]{BBD}. Considering the triangle $\jmath_!\jmath^!\ob{D} \to \ob{D} \to \imath_*\imath^*\ob{D} \to \jmath_!\jmath^!\ob{D}[1]$ for $\ob{D}\in \cat{D}^{\leq 0}$ shows that (2) and (4) are equivalent, and considering $\imath_*\imath^!\ob{D} \to \ob{D} \to \jmath_*\jmath^*\ob{D} \to \imath_*\imath^!\ob{D}[1]$ for $\ob{D}\in \cat{D}^{\geq 0}$ that (3) and (5) are equivalent.
\end{proof}

\subsection{Highest weight categories}\label{subsec:hwc}
Highest weight categories were introduced in \cite{Cline1988} as an abstraction of the representation theory of semisimple Lie algebras. We recall the definition in the special case of a Deligne finite category; see \cite{MR3742477} for a more general notion with coefficients in a division ring. 

An abelian category $\cat{A}$ is \defn{Deligne finite} if it is $\kk$-linear, $\mathrm{Hom}$ and $\mathrm{Ext}^1$-finite, length with finitely many iso-classes of simple objects, and has enough projective objects. It follows that $\cat{A}$ is equivalent to a module category over the endomorphism algebra of a projective generator, and therefore that $\cat{A}$ also has enough injectives and an injective cogenerator.

For a subset $\ob{A}$ of objects of $\cat{A}$ let $\Filt{\ob{A}}$ be the extension closure of the objects in $\ob{A}$.
Choose  a total ordering $\ob{S}_0,\ldots, \ob{S}_n$ of representatives of the iso-classes of simple objects in $\cat{A}$. The ordering  determines a filtration  $0=\cat{A}_{-1}\subset \cat{A}_0\subset\ldots\subset\cat{A}_n=\cat{A}$ by the Serre subcategories $\cat{A}_i = \Filt{ \ob{S}_0,\ldots,\ob{S}_i}$. Let $\ob{P}_i$ and $\ob{I}_i$ respectively be the projective cover and injective hull of $\ob{S}_i$ in $\cat{A}$.  The $i$-th \defn{standard object} $\stan{}{i}$ and \defn{costandard object} $\costan{}{i}$ are defined respectively as the maximal quotient of $\ob{P}_i$ and the maximal subobject of $\ob{I}_i$ in $\cat{A}_i$. These are respectively the projective cover and injective hull of $\ob{S}_i$ in $\cat{A}_i$.

The category $\cat{A}$ is \defn{highest weight with respect to the chosen ordering of its simple objects}  if $\EndA{\cat{A}}{\stan{}{i}}\cong\kk$ and $\ob{P}_i\in\Filt{\stan{}{i},\ldots,\stan{}{n}}$ for all $i\in\{0,\ldots,n\}$. We say $\cat{A}$ is \defn{highest weight} if it is so for some ordering. 

Note that we are assuming that the set of simple objects of $\cat{A}$ is totally ordered, but one could work with a partially ordered set of simple objects which is adapted to $\cat{A}$ in the sense of \cite{MR1211481}, see also \cite[\S5.2]{bodzenta} and \cite[Remark 2.4]{Cipriani:2024aa}.

Recall from \cite{MR3123754} that the inclusion $\cat{A}\hookrightarrow \cat{B}$ of a full abelian subcategory of an abelian category is  \defn{homological} if $\Ext{\cat{A}}{k}{\ob{A}}{\ob{A'}} \cong \Ext{\cat{B}}{k}{\ob{A}}{\ob{A'}}$ for all $\ob{A},\ob{A'}\in\cat{A}$ and $k\geq0$.

\begin{theorem}[{\cite[Theorem 3.4]{MR3742477}}]
\label{thm : criterion_hw}
A Deligne finite category $\cat{A}$ is highest weight if and only if there is a filtration $0=\cat{A}_{-1}\subseteq \cat{A}_0\subseteq \ldots \subseteq \cat{A}_n=\cat{A}$ by Serre subcategories such that the inclusions $\cat{A}_{r-1}\hookrightarrow \cat{A}_{r}$ are homological with $\cat{A}_{r}/\cat{A}_{r-1}\cong \vect{\kk}$ for $r=0,\ldots,n$.
\end{theorem}
The {\em global dimension}  of an abelian category $\cat{A}$ is 
\begin{equation*}
\gdim{\cat{A}}=\displaystyle\sup\{ d\in \N \mid  \Ext{\cat{A}}{d}{\sh{E}}{\sh{F}}\neq0,\ \mbox{for}\ \sh{E},\sh{F}\in\cat{A}\}
\end{equation*}
with $\gdim{\cat{A}}=\infty$ when the set above is unbounded. The definition does not require that $\cat{A}$ has enough projective objects, although when it does the global dimension is the maximal length of any minimal projective resolution. When $\cat{A}$ is a length category it is enough to check the vanishing of Ext-groups on pairs of simple objects. 

\begin{theorem}[{\cite{Cline1988} and  \cite[Statement 9]{MR987824}}]
The global dimension of a highest weight category $\cat{A}$ is finite, in fact $\gdim{A})\leq 2\rk{K(\cat{A})}-2$.
\end{theorem}

\subsection{Exceptional collections}
\label{subsec:exceptional}

Exceptional collections give rise to admissible thick subcategories and glued hearts, and are also closely related to highest weight categories. An object $\ob{E}\in \cat{D}$ is \defn{exceptional} if $\Ext{\cat{D}}{r}{\ob{E}}{\ob{E}} \cong \kk$ when $r=0$ and vanishes otherwise. A sequence $(\ob{E}_0,\ldots,\ob{E}_k)$ of exceptional objects is an \defn{exceptional collection} if $\ob{E}_j \in \orth \ob{E}_i$ whenever $i<j$.  An exceptional sequence is \defn{full} if it generates the triangulated category $\cat{D}$.

\begin{proposition}
\label{prop:except coll}
Suppose $(\ob{E}_0,\ldots,\ob{E}_k)$ is an exceptional collection in $\cat{D}$. Then there is a filtration by admissible thick subcategories 
\[
0 \subset \thick{}{\ob{E}_0} \subset \thick{}{\ob{E}_0,\ob{E}_1} \subset \cdots \subset \thick{}{\ob{E}_0,\ldots,\ob{E}_k} =:\cat{C} \subset \cat{D}
\]
with factors $\thick{}{\ob{E}_i} \simeq \cat{D}^b(\kk)$ for $i=0,\ldots,k$ and $\orth \cat{C}  \simeq  \cat{D} / \cat{C}\simeq \cat{C}\orth$.

Conversely, suppose $0 = \cat{D}_{-1} \subset \cat{D}_0 \subset \cdots \subset \cat{D}_k \subset \cat{D}$ is a filtration by admissible thick subcategories with quotients $\cat{D}_{i} / \cat{D}_{i-1} \simeq \der{b}{\kk}$ for $i=0,\ldots,k$. Then $({\jmath_0}_!\kk, \ldots,{\jmath_k}_!\kk)$ is an exceptional sequence in $\cat{D}$, where ${\jmath_i}_!$ is left adjoint to the quotient $\jmath_i^* \colon \cat{D}_{i} \to  \cat{D}_{i}/\cat{D}_{i-1}$.
\end{proposition}
\begin{proof}
The first statement is {\cite[Theorem 3.2]{Bondal_1990}} and the second follows from the adjunction and the fact that $\kk$ is exceptional in $\der{b}{\kk}$.
\end{proof}

When $(\ob{E}_0,\ldots,\ob{E}_n)$ is a full exceptional collection in $\cat{D}$ there is a unique heart glued from the hearts $\Filt{\ob{E}_i} \subset \thick{}{\ob{E}_i}$ with respect to the filtration 
\[
0 \subset \thick{}{\ob{E}_0} \subset \thick{}{\ob{E}_0,\ob{E}_1} \subset \cdots \subset \thick{}{\ob{E}_0,\ldots,\ob{E}_n} = \cat{D}.
\]
We say this heart is \defn{glued along} $(\ob{E}_0,\ldots,\ob{E}_n)$, see \cite[\S2.4]{Bodzenta:2026aa}. When $\Ext{\cat{D}}{<0}{\ob{E}_i}{\ob{E}_j}=0$ for $i< j$ the heart contains the exceptional collection. The next results describe particular cases. 

\begin{proposition}[{\cite[Lemma~3.14]{MR2335991}}]
\label{prop:macri}
 Suppose $(\ob{E}_0,\ldots,\ob{E}_k)$ is an exceptional collection with $\Ext{\cat{D}}{\leq 0}{\ob{E}_i}{\ob{E}_j}=0$ for all $i<j$. Then the heart in $\thick{}{\ob{E}_0,\ldots,\ob{E}_k}$ glued along $(\ob{E}_0,\ldots,\ob{E}_k)$ is 
 the extension closure $\Filt{\ob{E}_0,\ldots,\ob{E}_k}$ and is a length  heart  with simple objects $\ob{E}_0,\ldots,\ob{E}_k$.
\end{proposition}

The next result will be used in \cref{middle_perv_CP2} to give an example of a faithful glued heart which is not serially faithful.

\begin{proposition}
\label{prop:exceptional standards}
Suppose $\cat{H}$ is a Deligne finite heart in $\cat{D}$ with an ordering of its simple objects for which the standard objects form a full  exceptional collection $(\stan{}{0},\ldots,\stan{}{n})$. Then $\cat{H}$ is glued along $(\stan{}{0},\ldots,\stan{}{n})$.
\end{proposition}
\begin{proof}
Let $\ob{S}_0,\ldots,\ob{S}_n$ be (representatives of the iso-classes of) the simple objects of $\cat{H}$. Since $\stan{}{i}$ is exceptional there is a (unique) triangle
\[
\ob{K}_i \to \stan{}{i}\to \ob{S}_i \to \ob{K}_i[1]
\]
with $\ob{K}_i\in \Filt{\ob{S}_0,\ldots,\ob{S}_{i-1}}$. It follows by induction that $\thick{}{\stan{}{0},\ldots,\stan{}{i}} = \thick{}{\ob{S}_0,\ldots,\ob{S}_i}$  for each $i\in 0,\ldots,n$. Moreover,  $\cat{H} \cap \thick{}{\stan{}{0},\ldots,\stan{}{i}} = \Filt{\ob{S}_0,\ldots,\ob{S}_i}$ is a heart in $\thick{}{\stan{}{0},\ldots,\stan{}{i}}$ and we claim the semi-orthogonal decomposition $\langle \thick{}{\stan{}{i}}, \thick{}{\stan{}{0},\ldots,\stan{}{i-1}}\rangle$ restricts to its aisle. It suffices to check the claim for each simple object in the heart. For $\ob{S}_0,\ldots,\ob{S}_{i-1}$ it follows because they are in $\thick{}{\stan{}{0},\ldots,\stan{}{i-1}}$ and for $\ob{S}_i$ it follows from the above triangle. Hence $\cat{H}$ is glued with respect to the admissible thick filtration by \cref{prop: glued heart conditions}.
\end{proof}

\subsection{Faithful highest weight hearts}
\label{subsec:FH} 
A triangulated category is \defn{algebraic} if it admits a dg-enhancement in the sense of \cite{MR1055981}, see also \cite{Porta_2010} and \cite[Proposition 3.1]{Canonaco_2017}. A heart $\cat{H} \subset \cat{D}$ is \defn{faithful} if there is a realisation functor $\cat{D}^b(\cat{H}) \to \cat{D}$ from its bounded derived category which is an equivalence. Such a realisation functor always exists when $\cat{D}$ is an algebraic triangulated category, see \cite{keller-Vossieck, Letz-Sauter}. It is an equivalence if and only if $\Ext{\cat{H}}{m}{\ob{H}}{\ob{H'}} \cong \Ext{\cat{D}}{m}{\ob{H}}{\ob{H'}}$ for all $\ob{H},\ob{H'}\in \cat{H}$ and $m\in \Z$, see \cite[R\'emarque~3.1.17]{BBD} or \cite[Chapter 5, Theorem 3.7.3]{gelfandmanin}. 

We say a heart $\cat{H} \subset \cat{D}$ is \defn{serially faithful with respect to an admissible thick filtration} $0=\cat{D}_{-1} \subset \cat{D}_0 \subset \cdots \subset \cat{D}_n=\cat{D}$ if each $\cat{H} \cap \cat{D}_r$ is a faithful heart in $\cat{D}_r$. We say $\cat{H}$ is \defn{glued} if it is constructed inductively by glueing from hearts in $\cat{D}_{r+1}/\cat{D}_r$. When $n=0$ serially faithful reduces to faithful. We will be particularly interested in serially faithful hearts which are glued along full exceptional collections, \ie where each quotient of the thick admissible filtration is generated by a single exceptional object.  \cref{middle_perv_CP2} provides an example of a faithful heart which is glued along an exceptional collection but which is not serially faithful.

Let $(\ob{E}_0,\ldots,\ob{E}_n)$ be an exceptional collection in $\cat{D}$. Its \defn{left dual exceptional collection} is the exceptional sequence $(\ob{F}_n,\ldots,\ob{F}_0)$, with $\ob{F}_0=\ob{E}_0$ and 
\[
\ob{F}_i=L_{\ob{E}_0}\ldots L_{\ob{E}_{i-2}}L_{\ob{E}_{i-1}}\ob{E}_i
\]
 if $i>0$, where $L_{\ob{E}_j}\ob{E}_k$ denotes the \defn{left mutation of $\ob{E}_j$ at $\ob{E}_k$} defined by the canonical triangle
\[
\mor{*}{\ob{E}_j}{\ob{E}_k}\otimes \ob{E}_j \to \ob{E}_k \to L_{\ob{E}_j}\ob{E}_k\to \mor{*}{\ob{E}_j}{\ob{E}_k}\otimes \ob{E}_j[1].
\]
The left dual of a {\em full} exceptional collection is uniquely characterised by the Hom-duality property
\[
\Ext{\cat{D}}{d}{\ob{E}_i}{\ob{F}_j} \cong
\begin{cases}
\kk & i=j, d=0\\
0 & \text{otherwise}.
\end{cases}
\]
See \cite[\S2.3]{Bodzenta:2026aa} for further details.

Under the assumption that $\cat{D}$ is an algebraic triangulated category, faithful highest weight hearts can be characterised in terms of dual pairs of exceptional sequences as in \cite{Bodzenta:2026aa} or in terms of serial faithfulness. 
\begin{theorem}
\label{thm:fhw}
Let $\cat{D}$ be an algebraic triangulated category. The following are equivalent:
\begin{enumerate}
\item The heart $\cat{H}\subset \cat{D}$ is faithful and highest weight;
\item The sequence $(\stan{}{0},\ldots,\stan{}{n})$ of standard objects is a full exceptional collection with left dual pair the sequence $(\costan{}{n},\ldots,\costan{}{0})$ of  costandard objects;
\item The heart $\cat{H}$ contains a dual pair of full exceptional collections;
\item The heart $\cat{H}$ is  glued along a full exceptional collection and serially faithful for the associated admissible thick filtration.
\end{enumerate} 
\end{theorem}
\begin{proof} \leavevmode
$(1)\Rightarrow (2)$: Assume $\cat{H}\subset\cat{D}$ is faithful and highest weight. We prove that $(\stan{}{0},\ldots,\stan{}{n})$ is a full exceptional collection; the proof for $(\costan{}{n},\ldots,\costan{}{0})$ is dual.

The $i$-th standard object $\stan{}{i}$ is the projective cover of the simple $\ob{S}_i$ in $\Filt{\ob{S}_0,\ldots,\ob{S}_i}$. Since $\cat{H}$ is a faithful heart $\Ext{\cat{D}}{r}{\stan{}{i}}{\stan{}{j}}\cong \Ext{\cat{H}}{r}{\stan{}{i}}{\stan{}{j}} =0$ for $j\leq i$, unless $i=j$ and $r=0$. So the sequence of standard objects is exceptional. Now, assume that the simple object $\ob{S}_{i-1}\in\triang{}{\stan{}{0},\ldots,\stan{}{i-1}}$. The object $\ob{K}_i$ defined by the triangle 
\[
\ob{K}_i\to \stan{}{i}\to \ob{S}_i \to \ob{K}_i[1]
\]
is in $\Filt{\ob{S}_0,\ldots,\ob{S}_{i-1}}\subset \triang{}{\stan{}{0},\ldots,\stan{}{i-1}}$. Since $\ob{S}_0=\stan{}{0}$, this implies $\ob{S}_i\in\triang{}{\stan{}{0},\ldots,\stan{}{i}}$. Thus, the standard objects $(\stan{}{0},\ldots,\stan{}{n})$ of $\cat{H}$ form a full exceptional collection. The costandard sequence is left dual to the standard sequence by \cite[\S 2.3]{Bodzenta:2026aa}.

\noindent $(2)\Rightarrow (3)$: Immediate.

\noindent $(3)\Rightarrow (4)$: By \cite[Theorem 1.3 (1)]{Bodzenta:2026aa} the heart $\cat{H}$ is glued along the right dual of the pair exceptional collections. Then by \cite[Theorem 1.1 (2)]{Bodzenta:2026aa} it is serially faithful.

\noindent $(4)\Rightarrow (1)$:  Now suppose $\cat{H}\subset \cat{D}$ is  glued along a full exceptional collection and serially faithful for the associated admissible thick filtration. Then $\cat{H}$ is a length heart with $n+1$ simple objects and enough projectives and injectives by \cite[Proposition 4.8]{MR4349392}. Each $\cat{H}_r \coloneqq \cat{H} \cap \cat{D}_r$ is a Serre subcategory, and therefore is the extension closure of a subset of simple objects. Therefore  the chain $0=\cat{H}_{-1} \subset \cat{H}_0 \subset \cdots \subset \cat{H}_n=\cat{H}$ determines an ordering $\ob{S}_0,\ldots,\ob{S}_n$ of the simple objects. Moreover, since
\[
\Ext{\cat{H}_r}{*}{\ob{H}}{\ob{H'}} \cong \Ext{\cat{D}_r}{*}{\ob{H}}{\ob{H'}}
\cong \Ext{\cat{D}_{r+1}}{*}{\ob{H}}{\ob{H'}}   \cong \Ext{\cat{H}_{r+1}}{*}{\ob{H}}{\ob{H'}} 
\]
 for all $\ob{H},\ob{H'}\in \cat{H}_r$ the inclusions $\cat{H}_r \subset \cat{H}_{r+1}$ are homological. Hence $\cat{H}$ is a highest weight category for the given ordering by \cref{thm : criterion_hw}. Moreover $\cat{D}_r= \thick{}{\ob{S}_0,\ldots,\ob{S}_r}$  because the extension closure $\cat{H}_r$ of $\ob{S}_0,\ldots,\ob{S}_r$ is a bounded heart in $\cat{D}_r$.
\end{proof}

\begin{corollary}
Suppose $\cat{H}$ is a highest weight category. Then: 
\begin{enumerate}
\item  The sequence $(\stan{}{0},\ldots,\stan{}{n})$ of standard objects is a full exceptional collection in $\cat{D}^b(\cat{H})$ with left dual pair the sequence $(\costan{}{n},\ldots,\costan{}{0})$ of costandard objects;
\item  The category $\cat{H}$ is serially faithful and glued for the admissible thick filtration  $0=\cat{D}^b(\cat{H}_{-1}) \subset \cat{D}^b(\cat{H}_{0}) \subset \cdots \subset \cat{D}^b(\cat{H}_n)$ where $0=\cat{H}_{-1}\subset\cat{H}_0 \subset \cdots \subset \cat{H}_n=\cat{H}$ is the chain of Serre subcategories induced by the chosen ordering of simple objects. 
\end{enumerate}
\end{corollary}

\begin{remark}
Not every highest weight heart is serially faithful, or even faithful. For example the constructible sheaves on $\C\PP^1$ stratified by a point and its complement are highest weight but not faithful, see \cref{constr_sheaves_P1}.
\end{remark}

\section{Perverse sheaves}
\label{sec:perverse sheaves}

\subsection{Topologically stratified spaces} 
\label{sec: top strat}
The class of {\em topologically stratified spaces} was introduced in \cite{GMIH1,gm2}. The definition is inductive on the real dimension. A \defn{$0$-dimensional topologically stratified space} is a discrete union of points. A \defn{$d$-dimensional topologically stratified space} $X$ is a paracompact Hausdorff topological space endowed with a finite filtration by closed subsets of the form
\begin{equation*}
X=X_d\supset X_{d-1}\supset\ldots X_0\supset X_{-1}=\emptyset
\end{equation*}  
such that:
\begin{enumerate}[(i)]
    \item each difference  $X_i\smallsetminus X_{i-1}$ is a (possibly empty) topological $i$-manifold;
    \item each $x\in X_i\smallsetminus X_{i-1}$ has an open neighbourhood $V\subset X$  with a filtration preserving homeomorphism $V \cong \R^i \times C(L)$ where $L$ is a compact $(d-i-1)$-dimensional topologically stratified space.
 \end{enumerate}
The connected components of $X_i\smallsetminus X_{i-1}$ are the $i$-dimensional {\em strata} of $X$.  Let $\strata{X}$ be the set of strata. It is partially ordered by  $S\leq T \iff S\subset \overline{T}$. The \defn{depth} $\depth{X}$ is the height of this poset, \ie the length of a maximal chain in this partial order. 

 The stratified space $L=L_{d-i-1}\supset\cdots\supset L_0 \supset L_{-1}=\emptyset$ is the \defn{link} of the stratum containing $x$. The open cone $C(L) = L \times [0, 1)/L \times \{0\}$ is filtered by the $C(L_i)$ and the vertex $v$. The link of a stratum in a topologically stratified space need not be unique up to homeomorphism --- although it is for piecewise linear, {\it a fortiori} Whitney, stratified spaces --- see \cite[p128]{MR319207}. Let $S$ be a stratum and  $V$ be an open neighbourhood of $x\in S$ as above with  stratum preserving homeomorphism $V\to\R^{\dim(S)}\times C(L)$ mapping $V\cap S$ to $\R^{\dim(S)}\times \{v\}$ and $x$ to $(0,v)$. Let 
\begin{equation*}
\pi\colon V\smallsetminus S\cong\R^{\dim(S)}\times(C(L)\smallsetminus \{v\})\cong\R^{\dim(S)}\times(0,1)\times L\to L
\end{equation*}
be the projection and $\sigma\colon L\monic V\smallsetminus  S$ a choice of section. The strata of the link have the form $L_{S<T} \coloneqq \sigma^{-1}(T)$ for strata $T$ of $X$ with $S\subset\overline{T}$. We refer to these as the \defn{pairwise links}.

\subsection{Constructible derived categories} 
Let $\shf{X}$ be the abelian category of sheaves of $\kk$-vector spaces on a topologically stratified space $X$. The {\em constructible derived category} $\constr{c}{X}$ of $X$ is the full triangulated subcategory of the bounded derived category $\der{b}{X}=\der{b}{\shf{X}}$ consisting of constructible complexes. That is, the objects of $\constr{c}{X}$ are complexes of $\kk$-vector spaces on $X$ whose cohomology sheaves are locally-constant on each stratum of $X$. 

The constructible derived category $\constr{c}{X}$  is algebraic. This is because $\der{b}{X}=\der{b}{\shf{X}}$ admits a unique dg-enhancement by \cite[Theorem A]{Canonaco_2022} and any  full triangulated subcategory of an algebraic triangulated category is algebraic by \cite[Remark 3.2]{Canonaco_2017}.

Let $\imath\colon Z\hookrightarrow X$ be the inclusion of a closed union of strata and $\jmath\colon U=X\smallsetminus Z\hookrightarrow X$  the complementary open inclusion. Then extension by zero $\imath_*$ is a thick admissible inclusion and so by \cref{sec:thick subcats} determines a recollement 
\begin{equation}
\begin{tikzcd} [row sep={0.6cm},column sep={0.6cm}]\label{equation:6ff}
\constr{c}{Z}\ar[rr,"\imath_*","\perp"']&&
\constr{c}{X}\ar[ll,"\imath^!",bend left=35]\ar[ll,"\imath^*","\perp"', swap,bend right=50]\ar[rr,"\jmath^*","\perp"']&&
\constr{c}{U}.\ar[ll,"\jmath_*",bend left=35]\ar[ll,"\jmath_!","\perp"', swap,bend right=50]
\end{tikzcd}
\end{equation}
For brevity we abuse notation by denoting a functor and its left or right derived functor in the same way. The meaning will always be clear from the context. 

\subsection{Perversities and perverse sheaves}\label{subsec:perversities} 

Let $X$ be a $d$-dimensional topologically stratified space with finitely many strata. A {\em perversity on $X$} is a function $p\colon \strata{X} \to\Z$. The {\em dual perversity} is $p^*(S)=-\dim(S)-p(S)$ for any stratum $S\in \strata{X}$. 

A choice of total order on $\strata{X}$ refining the partial order determines a filtration of $X$ by the closed unions of strata $X_{T} \coloneqq \bigcup_{S\leq T} S$ for $T\in \strata{X}$. This in turn determines a filtration of $\constr{c}{X}$ by the admissible thick subcategories $\constr{c}{X_{T}}$. The successive quotients are the constructible derived categories $\constr{c}{S}$ for increasing $S\in \strata{X}$, where $S$ is considered as trivially stratified with one stratum. A perversity $p$ determines a bounded heart  $\loc{S}[-p(S)]$ in each $\constr{c}{X}$ consisting of shifted local systems. The $p$-perverse sheaves $\perv{p}{X}$  are the bounded heart in $\constr{c}{X}$ obtained by glueing these hearts \cite[\S 2.1.4]{BBD}.  More explicitly, the aisle and coaisle of the glued t-structure are
\begin{align*}
\pdl{p}{0} &= \{ \sh{E}\in\constr{c}{X} \ | \ \sh{H}^n(\imath_S^*\sh{E})=0 \ \ \mbox{if} \ \ n>p(S), \ \forall S\subset X\} \\
\pdg{p}{0} &= \{ \sh{E}\in\constr{c}{X} \ | \ \sh{H}^n(\imath_S^!\sh{E})=0 \ \ \mbox{if} \ \ n<p(S), \ \forall S\subset X\}, 
\end{align*}
where $\imath_S \colon S \hookrightarrow X$ is the inclusion and $\sh{H}^n\colon \constr{c}{S} \to \loc{S}$ the sheaf cohomology functor. 

The category $\perv{p}{X}$ is Deligne finite if and only if $X$ is a topologically stratified space with finitely many strata and the category of local systems on each stratum is Deligne finite. This is the case, for example, when each stratum has finite fundamental group \cite[Corollary 5.2]{MR4349392}. The above criterion is independent of the perversity.

We will only consider perversities which are related to the geometry of $X$, namely perversities which depend only on the dimension of the stratum and satisfy
\[
0 \leq p(S)-p(T) \leq \dim(T) - \dim(S),
\]
equivalently both $p$ and $p^*$ are decreasing functions of the dimension. In addition we normalise so that  $p(S)=0$ when $\dim(S)=0$.  We refer to these as \defn{Goresky--MacPherson (GM) perversities}. These are the perversities introduced in \cite{GMIH1} up to a shift of $-\dim(X)$. Our normalisation is the same as on \cite[p131]{dimca}, although in the latter only even-dimensional spaces are considered. 

Important examples of GM perversities include the \defn{zero perversity} $z(S)=0$, the \defn{top perversity} $t(S)=-\dim(S)$, and the \defn{upper and lower middle perversities} $n(S)=\left\lceil -\dim(S)/2 \right\rceil$ and $m(S)=\left\lfloor -\dim(S)/ 2\right\rfloor$ respectively. The zero and top perversity are dual, as are the upper and lower middle perversities.

\subsection{Faithful perversities}
\label{subsec:faithful perversities}

We review some cases in which the category of perverse sheaves is known to be faithful. We begin with the simplest situation: let  $X$ be a connected manifold trivially stratified with one stratum $X$. A perversity is an integer $p(X)$ and the $p$-perverse sheaves $\perv{p}{X} =  \loc{X}[-p(X)]$ are local systems placed in degree $p(X)$. When $X$ is $1$-connected this category is  equivalent to $\vect{\kk}$ and so has global dimension $0$. Therefore $\perv{p}{X}$  is faithful if and only if $\Ext{\constr{c}{X}}{i}{\kk_X}{\kk_X} \cong H^{i}(X; \kk)=0$ for all $i>0$, where $\kk_X$ is the constant sheaf on $X$ with stalk $\kk$. The case when $X$ is not $1$-connected is more complicated.

The fact that perverse sheaves on a contractible manifold are faithful has generalisations in topology and in complex geometry. The following topological generalisation was first observed by  MacPherson in unpublished lecture notes. 
\begin{theorem}[{\cite[Theorem~4.2]{MR1453053}}]
\label{thm:polischuk}
Let $X=|K|$ be the geometric realisation of a finite simplicial complex $K$ filtered by the skeleta, so that the strata are the interiors of the simplices of $K$. Then $\perv{p}{X} \subset \constr{c}{X}$ is faithful for any GM perversity $p$. 
\end{theorem}

In the context of complex geometry Beilinson \cite{MR923133} proved that the middle perverse sheaves are a faithful heart in the category $\constr{\text{alg-c}}{X}$ of all algebraically constructible, \ie constructible with respect to any stratification by algebraic subvarieties, sheaves on a complex variety $X$. If we fix a stratification, as in this paper, then we have the following result --- see also  \cite[Corollary~3.3.2]{MR1322847} for an alternative proof when $X$ is smooth and the strata are affine linear spaces.
\begin{theorem}[{\cite[\S1.5]{MR2119139}}]
\label{thm:bbm}
Set $\kk=\C$. Let $X$ be a complex variety stratified  by algebraic subvarieties so that each stratum $S$ is smooth with $H^{>0}(S;\C)=0$ and $S\hookrightarrow X$ an affine embedding. Then the middle perverse sheaves $\perv{m}{X}\subset\constr{c}{X}$ are faithful. 
\end{theorem}
\begin{remark}
\label{rmk:hw}
In the context of either of the above theorems it follows immediately that $\perv{p}{X}$ is serially faithful for the admissible thick filtration associated to any filtration of $X$ by closed unions of strata. Therefore $\perv{p}{X}$ is  highest weight by \cref{thm:fhw} because $\constr{c}{S} \simeq \cat{D}^b(\kk)$ for any stratum $S\subset X$ with vanishing higher cohomology.
\end{remark}

The following result is a common generalisation of \cref{thm:polischuk,thm:bbm}. We give a proof based on the results of \cite{Bodzenta:2026aa}. Many of the implications were previously known: \cite[Proposition 5.19]{ParshallScott} shows that conditions (2) and (3) below are equivalent, and \cite[Theorem 5.17]{ParshallScott} shows that (2) implies (1). Moreover, they show that (2) implies $\perv{p}{X}$ is highest weight (but not necessarily faithful) when the strata are $1$-connected with vanishing second cohomology. This latter condition is sufficient but not necessary, \eg the constructible sheaves on a space with $1$-connected strata are always highest weight, but only faithful when the higher cohomology of the pairwise links vanishes.

\begin{theorem}
\label{thm:fhwp}
Let $X$ be a topologically stratified space with finitely many contractible strata and denote by $\imath_S\colon S\to X$ the inclusion of a stratum into $X$.  The following are equivalent:
\begin{enumerate}
\item the category $\perv{p}{X}$ is faithful highest weight for some (hence any) total order refining the partial order of strata;
\item for any stratum $S\subset X$ the functors ${\imath_S}_*$ and ${\imath_S}_!$ are t-exact;
\item for any pair of distinct strata $S\subset\overline{T}$ the pairwise link $L_{S<T}$ satisfies 
\[
H^{>p(S)-p(T)}(L_{S<T})=0\ \text{and}\ H^{<p(S)-p(T)-1}_c(L_{S<T})=0.
\]
\end{enumerate}
\end{theorem}

\begin{proof}
\noindent $(1)\Leftrightarrow (2)$:
Suppose $\perv{p}{X}$ is faithful highest weight in $\constr{c}{X}$. Then by \cite[Theorem 1.3 and \S 2.4]{Bodzenta:2026aa} the standard objects are ${\imath_S}_!\kk_S[-p(S)]$, and the costandard objects are  ${\imath_S}_*\kk_S[-p(S)]$. Since the strata are assumed contractible this suffices to show that ${\imath_S}_!$ and ${\imath_S}_*$ are t-exact.

Now assume ${\imath_S}_!$ and ${\imath_S}_*$ are t-exact. We show that the standard and costandard objects form full exceptional sequences in $\constr{c}{X}$. First
\[
\Ext{\constr{c}{X}}{r}{\stan{}{S}}{\stan{}{S}} \cong \Ext{\constr{c}{X}}{r}{{\imath_S}_!\kk_S[-p(S)]}{{\imath_S}_!\kk_S[-p(S)]} \cong H^r(S;\kk)\cong\begin{cases} \kk \quad r=0 \\ 0 \quad r\neq 0 \end{cases}
\]
because the strata are contractible. Hence $\stan{}{S}$ is exceptional. Moreover $\Ext{\constr{c}{X}}{*}{\stan{}{S}}{\stan{}{T}}=0$ when $T<S$ since  ${\imath^*_S}\stan{}{T}=0$. Therefore, the standard objects $\stan{}{S}$ form an exceptional sequence. A dual argument applies to the costandard objects.

We use induction to prove these exceptional sequences are full by showing that all the simple objects are in the triangulated closure. The base case is $\ob{S}_0=\stan{}{0}$. Now suppose $\ob{S}_{i-1}\in\triang{}{\stan{}{0},\ldots,\stan{}{i-1}}$. There is a triangle 
\[
\ob{K}_i\to \stan{}{i}\to \ob{S}_i \to \ob{K}_i[1]
\]
with $\ob{K}_i\in\Filt{\ob{S}_0,\ldots,\ob{S}_{i-1}}$. Hence $\ob{S}_i\in\triang{}{\stan{}{0},\ldots,\stan{}{i}}$, completing the induction. 

Since for $S\neq T$ we either have ${\imath^*_S}\costan{}{T}=0$ or ${\imath^*_T}\stan{}{S}=0$ we conclude
\[
\Ext{\constr{c}{X}}{r}{\stan{}{S}}{\costan{}{T}}\cong \begin{cases} \kk \quad \hbox{if} \quad S=T \quad \hbox{and} \quad r=0 \\ 0 \quad \hbox{otherwise.} \end{cases}
\]
Therefore, by \cite[\S2.3]{Bodzenta:2026aa} the sequence of costandard objects is the left dual exceptional sequence to the one of standards. Since $\stan{}{S},\costan{}{T}\in\perv{p}{X}$ for any $S,T \subset X$ the heart $\perv{p}{X}$ is faithful and highest weight by \cite[Theorem 1.1]{Bodzenta:2026aa}.

\noindent $(2)\Leftrightarrow (3)$: 
Since $T$ is simply-connected the extension ${\imath_T}_*$ is t-exact if and only if ${\imath_T}_*\kk_T[-p(T)]$ is perverse. In general ${\imath_T}_*$ is left t-exact, so ${\imath_T}_*\kk_T[-p(T)]$ is perverse if and only if
\[ 
\sh{H}^r(\imath_S^*{\imath_T}_*\kk_T[-p(T)]) = 0 \quad \text{for}\ r>p(S) \ \text{and}\ S \subset \overline{T}.
\]
The equivalence then follows since  $\sh{H}^r(\imath_S^*{\imath_T}_*\kk_T[-p(T)])$ is a locally constant sheaf on $S$ with stalk $H^{r-p(T)}(L_{S<T})$ --- see \cref{cohomology links} below. A dual argument relates the exactness of ${\imath_T}_!$ to the other vanishing condition.
\end{proof}

\begin{corollary}
Suppose $X$ is a topologically stratified space with finitely many contractible strata and contractible pairwise links. Then $\perv{p}{X}$ is a faithful highest weight heart in $\constr{c}{X}$ for any GM-perversity $p$.
\end{corollary}
\begin{proof}
Since the links are contractible  $H^{>p(S)-p(T)}(L_{S<T})=0=H^{<p(S)-p(T)-1}_c(L_{S<T})$ because $0\leq p(S)-p(T) \leq \dim (L_{S<T})+1$. The result follows from \cref{thm:fhw}.
\end{proof}

There are highest weight perverse hearts which are not faithful, and  (non-perverse) faithful glued hearts in $\constr{c}{X}$ which are not highest weight.
\begin{example}
    Let $X=S^4$ stratified by a point and its complement.  The middle perverse sheaves $\perv{m}{S^4}$ are semi-simple, hence highest weight, but not faithful because
    \[
    0 = \Ext{\perv{p}{S^1}}{4}{\kk_{S^4}[2]}{\kk_{S^4}[2]} \neq 
    \Ext{\constr{c}{S^1}}{4}{\kk_{S^4}[2]}{\kk_{S^4}[2]} = H^4(S^4,\kk).  
    \]
\end{example}
 The following example, described from a more algebraic perspective, can be found in \cite[p284]{MR987824}. It shows that there are faithful glued hearts in a constructible derived category which have finitely many simple objects, enough projectives and finite  global dimension, but  which are not highest weight and therefore also not serially faithful for any ordering of their simple objects. However, the heart in the example is glued from an admissible thick filtration different from the geometric one, and therefore  is not a category of perverse sheaves for any perversity; we do not know of a perverse sheaf example with the above properties.
\begin{example}
\label{middle_perv_CP2}
Let $X=\C\PP^2\supset \C\PP^1\supset\C\PP^0$ be the projective plane stratified by a complete flag and $\cat{D}=\constr{c}{X}$ its constructible derived category. By \cite[p.217]{MR1862802}, the middle perverse sheaves $\cat{H}=\perv{m}{\C\PP^2}$ are equivalent to finite-dimensional modules  over the path algebra $\kk Q/I$  of the quiver  with relations 
\[
Q \coloneqq 
\begin{tikzcd}
0 \ar[r,bend left=30,"\alpha"] & 1 \ar[r,bend left=30,"\gamma"] \ar[l,bend left=30,"\beta"] &2 \ar[l,bend left=30,"\delta"]
\end{tikzcd}
\qquad 
,
\qquad
I=\langle \alpha\gamma, \delta\beta, \delta\gamma, \beta\alpha-\gamma\delta\rangle.
\]
 The heart $\cat{H}$ is highest weight for the ordering of the simple objects by increasing dimension of support, corresponding to the vertex ordering $0< 1< 2$ in the quiver description, and has global dimension four. By  \cref{thm:fhw} it is serially faithful  for the admissible thick filtration $0 \subset \constr{c}{\C\PP^0} \subset \constr{c}{\C\PP^1} \subset \constr{c}{\C\PP^2}$. 
 
The category $\cat{D}$ has another faithful, global dimension four heart. Let $\ob{S}_i$, $\ob{P}_i$, $\Delta_i$ and $\nabla_i$ denote respectively the $i$th simple, its projective cover, the $i$th standard and $i$th costandard in $\perv{m}{X}$ with respect to the above vertex ordering. The  perverse sheaf $\sh{E} = \ob{P}_0\oplus \ob{P}_1\oplus \Delta_1$ is a tilting object in $\cat{D}$. It determines an equivalence $\cat{D} \simeq \cat{D}^b(\mod{\End{\sh{E}}})$. Let $\cat{H}'$ be the faithful heart in $\cat{D}$ corresponding to the module heart under this equivalence, \ie the full subcategory on $\{ \sh{F}\in \cat{D} \mid \Ext{\cat{D}}{\neq 0}{\sh{E}}{\sh{F}}=0\}$. Alternatively, $\cat{H}'$ is the Happel--Reiten--Smal\o\ tilt of $\cat{H}$ at the torsion-free subcategory generated by $\ob{S}_2$. This is faithful by \cite[Theorem A]{MR3989131}. 

The heart $\cat{H}'$ is equivalent to finite dimensional modules over the path algebra $\kk Q'/I'$ of the quiver with relations
\[
Q'\coloneqq
\begin{tikzcd}[row sep=small]
& 1' \ar[dr,"\beta'"] \\
0'\ar[ur,"\alpha'"]  & & 2' \ar[ll,"\gamma'"] 
\end{tikzcd}
\qquad 
,
\qquad
I'=\langle \alpha'\beta'\gamma'\alpha', \gamma'\alpha'\beta'\rangle.
\]
It has global dimension four and is not highest weight for the ordering $0' < 1' < 2'$ (nor for any other ordering). The standard objects for the above ordering form an exceptional collection  $(\ob{S}_0',\ob{S}_1',\ob{P}_2')$  in $\cat{D}$ so that 
\[
0\subset \thick{}{\ob{S}_0'} \subset \thick{}{\ob{S}_0',\ob{S}_1'} \subset \thick{}{\ob{S}_0',\ob{S}_1',\ob{P}_2'}=\cat{D}
\]
is an admissible thick filtration with factors $\cat{D}^b(\kk)$ by \cref{prop:except coll}. Moreover, $\cat{H}'$ is glued with respect to this filtration by \cref{prop:exceptional standards}. Nevertheless, $\Filt{\ob{S}_0',\ob{S}_1'}$ is not a faithful heart in $\thick{}{\ob{S}_0',\ob{S}_1'}$ because  $\Ext{\cat{H}'}{2}{\ob{S}_0'}{\ob{S}_1'}\cong 0 \neq \kk \cong \Ext{\cat{D}}{2}{\ob{S}_0'}{\ob{S}_1'}$. Therefore $\cat{H}'$ is a faithful glued heart which is not serially faithful.

Indeed, since $\cat{H}'$ is not highest weight for any ordering of its simple objects \cref{thm:fhw} shows that it cannot be a serially faithful glued heart for {\em any} admissible thick filtration with quotients $\cat{D}^b(\kk)$.  

One can check that the costandard objects do not form an exceptional sequence, indeed $\nabla_{1'}$ is not exceptional. Thus the criterion of \Cref{thm:fhw} for the heart to be faithful highest weight fails. 
\end{example}

Finally we note there are faithful perverse hearts which are not equivalent to module categories over finite dimensional algebras. The assumption that the strata have no higher cohomology is not necessary, nor is it necessary for $\perv{p}{X}$ to have enough projectives, equivalently for the strata to have finite fundamental groups.
\begin{example}
    Let $\kk=\C$ and $X=\C \supset \{0\}$ with $\imath \colon \{0\} \to X$ and $\jmath\colon \C^*\to X$ the complementary closed and open inclusions.  For $M\in GL_n(\C)$ let $\sh{L}(M)$ be the local system on $\C^*$ with stalk $\C^n$ and monodromy $M$. The category $\loc{\C^*}$ of local systems, {\it a fortiori} the category $\perv{p}{X}$ of perverse sheaves, does not have enough projectives for any perversity $p$ since $\pi_1(U)\cong\mathbb Z$. The constructible, or zero perverse,  sheaves are not faithful, but the middle perverse sheaves are.
        \begin{enumerate}
        \item The simple objects of the zero perverse sheaves $\perv{z}{\C}$ are the extension by zero $\ob{S}_{\lambda}\coloneqq\jmath_! \sh{L}(\lambda)$ for $\lambda\in \C$ and $\ob{S}_0\coloneqq \imath_*\C_{\pt}$. Using the fact that  $\imath^! \ob{S}_{\lambda}\cong  \C[-1]\oplus \C[-2]$ when $\lambda=1$ and vanishes otherwise we compute
        	\[
	\Ext{}{1}{\ob{S}_{\lambda}}{\ob{S}_{\mu}}\cong\begin{cases} \C \quad \hbox{if}\quad \lambda=\mu\neq0 \quad \hbox{or}\quad \lambda=0,\mu=1 \\ 0 \quad\hbox{otherwise} 
	\end{cases}
	\]
	and 
	$
	\Ext{\constr{c}{\C}}{2}{\ob{S}_0}{\ob{S}_1}\cong  \C.
	$
	In contrast $\Ext{\perv{z}{\C}}{2}{\ob{S}_0}{\ob{S}_1}=0$ because there is a perverse sheaf $\sh{E}$ such that the central diamond below is cartesian and cocartesian
	\[
	\begin{tikzcd}[row sep=small]
	&& \sh{E} \ar[->>]{dr} && \\
	& \jmath_!\sh{L}\left(\begin{smallmatrix}1&1\\0&1\end{smallmatrix}\right) \ar[->>]{dr} \ar[hook]{ur}&& \C_X \ar[->>]{dr}& \\
	\ob{S}_1 \ar[hook]{ur} && \ob{S}_1 \ar[hook]{ur} && \ob{S}_0
	\end{tikzcd}
	\]
	 Therefore $\perv{z}{\C}\subset\constr{c}{\C}$ is not faithful and has global dimension $1$.
        \item The simple objects of the middle perverse sheaves  $\perv{m}{\C}$ are $\ob{S}_0\coloneqq \C_0$ and the intermediate extensions
        \[
        \ob{S}_{\lambda}\coloneqq \jmath_{!*}\sh{L}(\lambda)=
        \begin{cases} 
        \C_X[1] & \lambda=1 \\ 
        \jmath_!\sh{L}(\lambda)[1] & \lambda\neq 1.
        \end{cases} 
        \]
        The only non-vanishing Ext groups are $\Ext{\constr{c}{\C}}{1}{\ob{S}_\lambda}{\ob{S}_{\mu}} \cong \C$ when $\lambda=0, \mu=1$ or $ \lambda=\mu\neq 0,1$ or $\lambda=1,\mu=0$. In particular, since all second Ext groups vanish, $\perv{m}{\C}\subset \constr{c}{\C}$ is faithful of global dimension $1$. It is evidently serially faithful, but it is not highest weight because the open stratum $\C^*$ is not simply-connected. 
    \end{enumerate}
\end{example}  

\subsection{A vanishing result}
\label{subsec:vanishing}

The following basic vanishing result for the hypercohomology of a perverse sheaf will be used in the next section to prove that faithful perverse hearts have global dimension bounded by the dimension of the underlying space. It does not seem easy to find this result in the literature in this generality. See \cite[Proposition 5.2.13 and Corollary 5.2.14]{dimca} for the case of analytic spaces.

\begin{proposition}
\label{prop:cohocptspace}
Let $X$ be a compact topologically stratified space and $\sh{E}\in\perv{p}{X}$. Then for any perversity for which both $p$ and $p^*$ are decreasing functions of the dimension, in particular for all GM-perversities, 
\begin{equation*}
H^k(X;\sh{E})=0 \ \text{for}\ 
\begin{cases} 
k<p(X) \\ 
k>-p^*(X).
\end{cases}
\end{equation*}
\end{proposition}

\begin{proof} 
Since $p$ is a GM-perversity $p^*(X)\leq p^*(S)$ and $p(X)\leq p(S)$ for all strata $S\subset X$. Therefore $\cat{D}^{<\dups(X)}(S)\subset \cat{D}^{<p^*(S)}(S)$ and $\cat{D}^{<p(X)}(S)\subset \cat{D}^{<p(S)}(S)$ for all strata $S\subset X$. Hence $\cat{D}^{<\dups(X)}(X)\subset {^{\dups}}\cat{D}^{<0}(X)$ and $\cat{D}^{<p(X)}(X)\subset {^{p}}\cat{D}^{<0}(X)$. 

Let $\pi\colon X \to \pt$ be the map to a point. By the above  $\pi^*\sh{F}\in{^{\dups}\cat{D}}^{<0}(X)$ whenever $\sh{F}\in\cat{D}^{<\dups(X)}(\pt)$ and by duality $\pi^!\sh{F}\in{}^{p}\cat{D}^{>0}(X)$ whenever $\sh{F}\in\cat{D}^{>-p^*(X)}(\pt)$. Similarly, $\pi^*\sh{G}\in{}^{p}\cat{D}^{<0}(X)$ whenever $\sh{G}\in\cat{D}^{<p(X)}(\pt)$. 

For compact $X$ the functor $\pi_*$ is left adjoint to $\pi^!$ and right adjoint to $\pi^*$. Therefore  $\Mor{}{\pi_*\sh{E}}{\sh{F}}\cong\Mor{}{\sh{E}}{\pi^!\sh{F}}=0$ whenever $\sh{F}\in\cat{D}^{>-p^*(X)}(\pt)$ and similarly $\Mor{}{\sh{G}}{\pi_*\sh{E}}\cong\Mor{}{\pi^*\sh{G}}{\sh{E}}=0$
 whenever $\sh{G}\in\cat{D}^{<p(X)}(\pt)$. Thus 
 \[
 \pi_*\sh{E}\in  \cat{D}^{\geq p(X)}(\pt) \cap \cat{D}^{\leq -p^*(X) }(\pt).
 \]
  The claimed vanishing for $H^k(X;\sh{E})=H^k(\pi_*\sh{E})$ follows immediately.
\end{proof}

\subsection{Upper bounds on global dimension}
\label{section:GD} 

Let $X$ be a topologically stratified space with finitely many strata. In general $\perv{p}{X}$ may have infinite global dimension, see \cref{ex:infinite gldim}. When $\perv{p}{X}$ is faithful  we show its global dimension is bounded by $\dim X$.

The proof will be by induction on the depth. For the inductive step we need a vanishing result for the restriction of a perverse sheaf to the link of a stratum which we now prove. 

Let $\imath \colon S\monic X$ be the inclusion of a closed stratum and $\jmath\colon X \smallsetminus S\monic X$ the complementary open inclusion. Let $\pi \colon V\smallsetminus S \to L$ and $\sigma \colon L \to V \smallsetminus S$ be the projection from a deleted distinguished neighbourhood of $x$ to the link and a section of it respectively --- see \cref{sec: top strat}. Note that $\sigma$ is a normally nonsingular inclusion of codimension $\dim(S)+1$, see \cite[\S 5.4]{gm2}. 

The strata of the link have the form $L_{S<T} \coloneqq \sigma^{-1}(T)$ for strata $T$ of $X$ with $S\subset\overline{T}$. 
Given a perversity $p$ on $X$ we define a  perversity $\pbar$ on the link $L$ by  $\pbar (L_{S<T}) \coloneqq p(T)$. The following lemma shows this  perversity is a GM-perversity up to an overall shift.
\begin{lemma}
If $p$ and $p^*$ are decreasing functions of the dimension then so are $\pbar$ and $\pbar^*$. 
\end{lemma}
\begin{proof} 
This is clear for $\pbar$. For the dual it follows from 
\begin{align*}
\pbar^*(L_{S<T}) 
&=-\pbar(L_{S<T})-\dim(L_{S<T})\\
&=-p(T)-\dim(T)+\dim(S)+1\\
&= p^*(T)+\dim(S)+1.\qedhere
\end{align*} 
\end{proof}
 
The \defn{restriction of $\sh{E}\in\constr{c}{X}$ to $L$}  in $\constr{c}{L}$ is the pullback $\sh{E}|_{L} \coloneqq \sigma^*(\sh{E}|_{V\smallsetminus S})$ along the section of the restriction of $\sh{E}$ to the open subset $V\smallsetminus S$.

\begin{lemma}
The isomorphism class of $\sh{E}|_{L}$ is independent of the choice of the section $\sigma$.
\end{lemma}
\begin{proof}
Since the restriction $\sh{E}|_{V\smallsetminus S}$ is cohomologically constant on the fibres of $\pi$, which are contractible, the counit $\pi^*\pi_*(\sh{E}|_{V\smallsetminus S})\to \sh{E}|_{V\smallsetminus S}$ of the adjunction is an isomorphism in $\constr{c}{V\smallsetminus S}$ by \cite[\S1.13.(17)]{gm2}. Thus 
\[
\sh{E}|_{L} \coloneqq \sigma^*(\sh{E}|_{V\smallsetminus S})
 \cong   \sigma^*\pi^*\pi_*(\sh{E}|_{V\smallsetminus S}) = \pi_*(\sh{E}|_{V\smallsetminus S})
 \]
in $\constr{c}{L}$ because $\pi\sigma=\id$. The right hand expression is independent of the choice of $\sigma$. 
\end{proof}

The next result shows that the groups $H^k(L;\sh{E}|_{L})$ are independent of the choice of $L$ up to isomorphism. The Whitney stratified case can be found in  \cite[Lemma 9.3]{goresky}. 

\begin{lemma}
\label{cohomology links}
The hypercohomology $H^k(L;\sh{E}|_{L}) \cong \sh{H}^k_{x}(\imath^*\jmath_*\jmath^*\sh{E})$ for all $k\in\Z$.
\end{lemma}
\begin{proof} 
We have $H^k(L;\sh{E}|_{L})\cong H^k(L;\pi_*(\sh{E}|_{V\smallsetminus S}))\cong H^k(V\smallsetminus S; \sh{E}|_{V\smallsetminus S}) \cong H^k(V;\jmath_*\jmath^*\sh{E}|_{V})$. Since $x$ has a cofinal sequence of neighbourhoods of the form $V$, it follows that
\begin{equation*}
 H^k(L;\sh{E}|_{L}) \cong \sh{H}^k_{x}(\imath^*\jmath_*\jmath^*\sh{E})
\end{equation*}
for all $k\in\Z$. 
\end{proof}

\begin{lemma}
\label{shifted loc sys}
\label{cor:link vanishing}
If $\sh{F}\in\perv{p}{X}$ then $\sh{F}|_{L} \in \perv{\pbar}{L}$. Hence
\[
H^k(L;\sh{F}|_{L})=0 
\]
if $k<p(X)$ or $k >p(X)+\codim(S)-1$.
\end{lemma}

\begin{proof}
Let $\sh{F}\in \perv{p}{X}$. For a stratum $t\colon T\monic X$ with $S\subset \overline{T}$ let   $t_L\colon L_{S<T} \monic L$ be the inclusion. Then 
\[
t_L^*(\sh{F}|_{L}) = t_L^*\sigma^*(\sh{F}|_{V\smallsetminus S})\cong\sigma^*t^*(\sh{F}|_{V\smallsetminus S}).
\]
Therefore $t_L^*(\sh{F}|_{L})\in\cat{D}^{\leq-p(T)}(L_{S<T})$ because  $t^*(\sh{F}|_{V\smallsetminus S})\in\cat{D}^{\leq-p(T)}(T\cap V)$ and $\sigma$ is normally nonsingular. Using the normal nonsingularity of $\sigma$ again we have
\begin{align*}
t_L^!(\sh{F}|_{L}) 
& = t_L^!\sigma^*(\sh{F}|_{V\smallsetminus S}) \cong t_L^!\sigma^!(\sh{F}|_{V\smallsetminus S})[\dim(S)+1]\\ 
&\cong\sigma^! t^!(\sh{F}|_{V\smallsetminus S})[\dim(S)+1]\cong\sigma^* t^!(\sh{F}|_{V\smallsetminus S}).
\end{align*}
Since $t^!(\sh{F}|_{V\smallsetminus S})\in \cat{D}^{\geq-p(T)}(T\cap V)$ it follows that $t_L^!(\sh{F}|_{L})\in\cat{D}^{\geq-p(T)}(L_{S<T})$. We conclude that $\sh{F}|_{L}\in\perv{\pbar}{L}$ as claimed.

The vanishing results follow from \cref{prop:cohocptspace} and the fact that  $\pbar(L)=\dim(X)$ and $-\pbar^*(L)= \pbar(L)+\dim(L)=p(X)+\codim(S)-1$.
\end{proof}

We now bound the global dimension of a faithful heart by the dimension of $X$, see also  \cite[\S 5.1]{MR4349392}. In the next section we give conditions under which the bound is achieved. 

\begin{theorem}
\label{on gl dim}
Let $X$ be a topologically stratified space with finitely many strata and $p$ a GM-perversity. For any $\sh{E},\sh{F}\in\perv{p}{X}$ we have
\begin{equation*}
\Ext{\constr{c}{X}}{k}{\sh{E}}{\sh{F}}=0\ \ \mbox{if}\ k<0 \ \text{or}\  k>\dim(X).  
\end{equation*}
In particular, if $\perv{p}{X}$ is faithful then $\gdim{\perv{p}{X}}\leq\dim(X)$.
\end{theorem}

\begin{proof} 
We use induction on the number of strata.  When $X$ has a single stratum, perverse sheaves  are shifted local systems and  $\Ext{\constr{c}{X}}{k}{\sh{E}}{\sh{F}}\cong\Ext{\constr{c}{X}}{k}{\kk_X}{\sh{E}^{\vee}\otimes\sh{F}}=0$ for $k<0$ or $k>\dim(X)$ by \cref{prop:cohocptspace}.

For the inductive step let $\imath \colon S \monic X$ be the inclusion of a closed stratum and $\jmath \colon  X\smallsetminus S \monic X$ the complementary open inclusion. Apply $\Ext{\constr{c}{X}}{k}{\sh{E}}{-}$ to the triangle
\begin{equation}\label{equation:triangle1}
\imath_*\imath^!\sh{F}\to\sh{F}\to\jmath_*\jmath^*\sh{F}\to\imath_*\imath^!\sh{F}[1]\end{equation}
to obtain the long exact sequence
\begin{equation}
\label{LES ext gdim}
\cdots \to \Ext{\constr{c}{X}}{k}{\sh{E}}{\imath_*\imath^!\sh{F}}\to \Ext{\constr{c}{X}}{k}{\sh{E}}{\sh{F}}\to \Ext{\constr{c}{X}}{k}{\sh{E}}{\jmath_*\jmath^*\sh{F}}\to\cdots
\end{equation}
By adjunction, the right hand term in (\ref{LES ext gdim}) is $\Ext{\constr{c}{U}}{k}{\jmath^*\sh{E}}{\jmath^*\sh{F}}$ which vanishes for $k<0$ and $k>\dim X$ by induction. To complete the inductive step we show that the left hand term, which by adjunction is $\Ext{\constr{c}{S}}{k}{\imath^*\sh{E}}{\imath^!\sh{F}}$,  vanishes in the same degrees. 

Applying $\imath^*$ to (\ref{equation:triangle1}) and taking cohomology yields the long exact sequence in $\loc{S}$
\begin{equation*}
\cdots \to \sh{H}^{k}(\imath^!\sh{F})\to \sh{H}^{k}(\imath^*\sh{F})\to \sh{H}^{k}(\imath^*\jmath_*\jmath^*\sh{F})\to \sh{H}^{k+1}(\imath^!\sh{F})\to \cdots.
\end{equation*}
 By \cref{cohomology links} and  \cref{shifted loc sys},  $\sh{H}^{k}(\imath^*\jmath_*\jmath^*\sh{F})=0$  for $k<p(X)$ and $k>p(X)+\codim(S)-1$.  Combined with the fact that $\imath^!$ is left t-exact we deduce that $\sh{H}^k(\imath^*\sh{F})=0$ for $k<p(X)$ or $k>0$, and similarly using the fact that $\imath^*$ is right t-exact we have $\sh{H}^k(\imath^!\sh{F})=0$ for $k<0$ or $k>p(X)+\codim(S)$. This applies to any perverse sheaf, in particular to $\sh{E}$ in place of $\sh{F}$.  It follows immediately that 
 \[
 \Ext{\constr{c}{S}}{k}{\imath^*\sh{E}}{\imath^!\sh{F}} = 0
 \]
  for $k<0$. By applying the vanishing result inductively over the cohomological filtrations of $\imath^!\sh{E}$ and $\imath^*\sh{F}$ we also conclude that 
 \[
 \Ext{\constr{c}{S}}{k}{\imath^*\sh{E}}{\imath^!\sh{F}} = 0
 \]
 for $k>\dim{S}+\codim{S}=\dim(X)$ as required. 
\end{proof}
Theorem~\ref{on gl dim} was known for a complex analytic space with a complex analytic Whitney stratification with $\pi_1(S)=\pi_2(S)=0$ for all but the open strata, see \cite[Theorem 1.3]{MR925719}. A bound for the global dimension of the category of middle perverse sheaves on a complex algebraic variety is also obtained in \cite[Corollary 3.2.2]{MR1322847} as a consequence of the fact that the middle perverse sheaves are highest weight.  

We provide some illustrative examples. The first shows the bound on the global dimension can be strict. The second is an example of a perverse heart which has infinite global dimension, and therefore is not faithful. The third shows that the converse of \cref{on gl dim} is false.

\begin{example}
Let $X=\R^n$ trivially stratified. Then for any perversity $\perv{p}{X} \simeq \vect{\kk}$ is faithful of global dimension $0$ whereas $\dim X=n$. 
\end{example}
\begin{example}
\label{ex:infinite gldim}
Consider $X=\C\PP^2$ stratified by $\C\PP^1$ and its complement $\C^2$ and let $m$ be the middle perversity. By \cite[Example 6.3]{MR833195}, the category of middle perverse sheaves on $X$ is equivalent to representations of the quiver 
\begin{center}
\begin{tikzpicture}
\node (A) at (-1,0) {$1$};
\node (B) at (1,0) {$2$};
\path[->,font=\scriptsize,>=angle 90]
(B) edge [bend left] node[below] {b} (A);
\path[->,font=\scriptsize,>=angle 90]
(A) edge [bend left] node[above] {a} (B);
\end{tikzpicture}
\end{center}
\noindent with ideal of relations $I=\langle ab, ba\rangle$. Since the global dimension of $\perv{m}{X}$ is infinite, this heart cannot be faithful by Theorem~\ref{on gl dim}.
\end{example}
\begin{example}\label{constr_sheaves_P1}
    Let $X=\C\PP^1$ stratified by a point and its complement. The category $\perv{z}{X}$ of zero-perverse or constructible sheaves is equivalent to representations of the $A_2$ quiver by \cite{MR2575092}. Thus, the inequality of Theorem~\ref{on gl dim} becomes
    \[
    \gdim{\perv{z}{\C\PP^1}}=1 < 2=\dim(\C\PP^1).
    \]
    However $\perv{z}{X}\subset \constr{c}{X}$ is not faithful  since $\Ext{\constr{c}{X}}{2}{\kk_X}{\kk_X}\cong H^2(X)\cong\kk$.
\end{example}

When a perverse heart is faithful highest weight there is an alternative bound on the global dimension coming from the depth of the stratification.

\begin{corollary}
\label{cor: gldim depth bound}
    Let $X$ be a topologically stratified space with finitely many strata, each $1$-connected. Then   $\gdim{\perv{p}{X}}\leq 2 \depth{X}$ when $\perv{p}{X}$ is faithful highest weight. 
\end{corollary}
\begin{proof}
Since $\perv{p}{X}$ is highest weight for any total order refining the partial order of strata, $\gdim{\perv{p}{X}} \leq 2 \depth{X}$ by \cite[Lemma~2.2]{MR1211481}. 
\end{proof}

\subsection{Lower bounds on  global dimension}
\label{satisfaction_gd_bounds}

When a shift of the constant sheaf is perverse we obtain a lower bound on the global dimension of a faithful perverse heart in terms of the cohomology of the space.

First we recall conditions under which a shift of the constant sheaf is perverse. This is immediate when $p$ is the zero perversity, but occurs in  other situations too.
\begin{lemma}
\label{lem:shifted constant perverse}
The shifted constant sheaf $\kk_X[-p(X)]$ is perverse if $X$ is
\begin{enumerate}
    \item a topological manifold and $p$ is any GM-perversity or
    \item a complex analytic local complete intersection and $p$ is the middle perversity.
\end{enumerate}
In the first case $\kk_X[-p(X)]$ is simple if $p(S)>p(X)$ for $\dim S < \dim X$, and in the second case it is always simple. 
\end{lemma}
\begin{proof}
The first case is a direct calculation and the second is \cite[Theorem 5.1.20]{dimca}.
\end{proof}

\begin{corollary}
\label{cor:equality_gldim}
Let $X$ be a topologically stratified space and $p$ a GM-perversity for which $\perv{p}{X}$ is faithful and $\kk_X[-p(X)]$ perverse. Then 
\[
\gdim{\perv{p}{X}} \geq \max \{ d \mid H^d(X;\kk) \neq 0\}.
\] 
\end{corollary}
\begin{proof}
Since the heart is faithful,
    \begin{align*}
        \Ext{\perv{p}{X}}{d}{\kk_X[-p(X)]}{\kk_X[-p(X)]}&\cong\Ext{\constr{c}{X}}{d}{\kk_X}{\kk_X}\cong H^d(X;\kk). \qedhere
    \end{align*}
  \end{proof}
  
  \begin{example}
  \label{ex:manifold gldim eq}
  If $X$ is a compact, $\kk$-oriented topologically stratified manifold and $p$ is a perversity for which $\perv{p}{X}$ is faithful then $\gdim{\perv{p}{X}}=\dim(X)$. This follows immediately from the above together with the fact that $H^{\dim(X)}(X;\kk)\neq 0$.
  \end{example}
  
\begin{corollary}
\label{cor:dim bded by depth}
Let $X$ be a topologically stratified compact $\kk$-orientable topological manifold with finitely many strata, each $1$-connected.
Then $\dim(X) \leq 2 \depth{X}$  when $\perv{p}{X}$ is   serially faithful.
\end{corollary}
\begin{proof}
By  \cref{ex:manifold gldim eq} and \cref{cor: gldim depth bound}, $\dim(X)= \gdim{\perv{p}{X}} \leq 2 \depth{X} $. 
\end{proof}

\begin{example}\label{exe:Pn_gldim} Let $X=\C\PP^n$ with affine stratification. The middle perverse sheaves  $\perv{m}{X}$ are faithful by \cref{thm:bbm} and so have global dimension $\dim \C\PP^n=2n$ by \cref{ex:manifold gldim eq}. Note that $\dim(X) = 2 \depth{X}$ in conformity with \cref{cor:dim bded by depth}; moreover any stratification of lesser depth cannot have $1$-connected strata.
\end{example}

\subsection{Hypercohomology of perverse sheaves}
\label{hypercoho_computations}

Suppose $X$ is a topologically stratified space with finitely many strata, each with finite fundamental group, and that $p$ is a perversity for which $\perv{p}{X}$ is a faithful heart in $\constr{c}{X}$. The fundamental group assumption implies that  $\perv{p}{X}$ has enough projective and injective objects and finite global dimension by \cref{on gl dim}. Thus there are equivalences
\[
\constr{c}{X} \simeq \der{b}{\perv{p}{X}} \simeq \cat{K}^b(\mathrm{Proj}\, \perv{p}{X} ) \simeq \cat{K}^b(\mathrm{Inj}\, \perv{p}{X} )
\]
where $\mathrm{Proj}\, \perv{p}{X}$ and $\mathrm{Inj}\, \perv{p}{X}$ are respectively the exact subcategories of projective and injective perverse sheaves. 

These algebraic incarnations of the constructible derived category provide tools for computing the hypercohomology of perverse sheaves. Let $\sh{P}^*$ be a projective resolution of the constant sheaf $\kk_X$ and $\sh{I}^*$ an injective resolution of the dualising complex $\sh{D}_X$. 

\begin{corollary}
\label{prop:hypercoh}
Suppose $X$ is a topologically stratified space with finitely many strata, each with finite fundamental group, and that $p$ is a perversity for which $\perv{p}{X}$ is a faithful heart in $\constr{c}{X}$. Then  $\pi_*\sh{E} \cong  \mor{}{\sh{P}^*}{\sh{E}}$ and $\pi_!\sh{E} \cong  \mor{}{\sh{E}}{\sh{I}^*}^\vee$ in $\constr{c}{\pt}$ where $\pi \colon X \to \pt$ is the map to a point.
\end{corollary}
\begin{proof}
The two statements are proved by similar direct calculations. It is enough to show that the cohomology groups are isomorphic. For the first
\[    
  H^r(\pi_*\sh{E}) 
 \cong \Ext{\constr{c}{X}}{r}{\kk_X}{\sh{E}}
 \cong \Ext{\der{b}{\perv{p}{X}}}{r}{\kk_X}{\sh{E}}
 \cong H^{r}( \mor{}{\sh{P}_*}{\sh{E}} ).\qedhere
\]
\end{proof}

\begin{corollary}\label{cor:computation_cohom}
Let $\sh{E}$ be a simple perverse sheaf. Suppose in addition that $\kk_X$ is a shifted perverse sheaf and $\sh{P}^*$ its minimal projective resolution. Then  $H^r(X,\sh{E}) \cong \kk^a$ where $a$ is the multiplicity of the projective cover of $\sh{E}$ as a summand of the term $\sh{P}^{-r}$. 

Dually, if $\sh{D}_X$ is shifted perverse and $\sh{I}^*$ its minimal injective resolution then $H^r_c(X,\sh{E}) \cong \kk^b$ where $b$ is the multiplicity of the injective hull of $\sh{E}$ as a summand of the term $\sh{I}^{r}$.
\end{corollary}
\begin{proof}
This follows from \cref{prop:hypercoh} and the fact that the complex $ \mor{}{\sh{P}^*}{\sh{E}}$ has vanishing differential when $\sh{E}$ is simple \cite[proof of Corollary 2.5.4]{MR1110581}. To see this recall that the minimal projective resolution is constructed by taking $\sh{P}^0$ to be the projective cover of the top of $\kk_X[-p(X)]$, then taking $\sh{P}^{-1}$ to be the projective cover of the top of the kernel of $\sh{P}^0 \to \kk_X[-p(X)]$ and so on. Therefore composing with $\sh{P}^{i-1} \to \sh{P}^i$ induces the zero map
\[
\mor{}{\sh{P}^i}{\sh{E}} \to \mor{}{\sh{P}^{i-1}}{\sh{E}}
\]
when $\sh{E}$ is simple. The dual result is proved similarly.
\end{proof}

The following is a simple example of its application in the complex algebraic context. In the next section we apply it in the case of simplicial complexes.
\begin{example}
Consider the middle perverse sheaves $\perv{m}{\C\PP^2}$ on $\C\PP^2$ with affine stratification as in \cref{middle_perv_CP2}. As there, let $\ob{S}_r = \kk_{\C\PP^r}[r]$ denote the simple perverse sheaf supported on $\C\PP^r$ and $\ob{P}_r$ its projective cover. The projective resolution of $\kk_X = \ob{S}_2[-2]$ is 
\[
\begin{tikzcd}
0\ar{r} 
&  \ob{P}_2 \ar{r}{\delta} 
& \ob{P}_1 \ar{r}{\begin{pmatrix} \beta \\-\gamma \end{pmatrix}}
& \ob{P}_0\oplus \ob{P}_2 \ar{r}{(\alpha,\delta)}  
& \ob{P}_1\ar{r}{\gamma}
& \ob{P}_2 \ar{r} 
& 0.
\end{tikzcd}
\]
where the direct sum term is in degree zero and we identify irreducible morphisms between the projectives with the arrows in the quiver presentation in \cref{middle_perv_CP2}. By \cref{cor:computation_cohom}, taking morphims into perverse sheaves yields complexes computing their cohomology. For example taking morphisms into $S_2$ yields a split complex computing the (shifted) cohomology
\[
H^r(\C\PP^2;\ob{S}_2)=
\begin{cases} 
\kk \quad \hbox{if} \quad r=-2,0,2 
\\ 0 \quad \text{otherwise}
\end{cases}
\]
of $\C\PP^2$.  Taking morphisms into $\ob{P}_2=\stan{}{2}\cong\jmath_{!}\kk_{\C^2}[2]$  yields the complex 
\[
\begin{tikzcd}
0 \ar{r} 
& \langle \id_{\ob{P}_2} \rangle \ar{r}{\gamma} 
& \langle \gamma \rangle \ar{r}{0} 
&\langle \id_{\ob{P}_2} \rangle \ar{r}{\gamma} 
&\langle \gamma \rangle \ar{r}{0} 
&\langle \id_{\ob{P}_2} \rangle \ar{r}
&0
\end{tikzcd}
\]
computing the (shifted) compactly supported cohomology of $\C^2$. Finally, taking morphisms into $\ob{P}_0=\ob{I}_0$ yields the exact complex
\[
\begin{tikzcd}
0 \ar{r} 
& 0 \ar{r}
& \langle \beta \rangle \ar{r}{\alpha} 
&\langle \id_{\ob{P}_0},\alpha\beta \rangle \ar{r}{\beta} 
&\langle \beta \rangle \ar{r}
&0 \ar{r}
&0
\end{tikzcd}
\]
showing that $H^*(\C\PP^2;\ob{P}_0)=0$.
\end{example}

\subsection{Hypercohomology of perverse sheaves: the simplicial case}\label{hypercoho_simplicial}

Applying the results of \cref{hypercoho_computations} in the simplicial context recovers the theory of cellular perverse sheaves \cite{goresky} and the complexes for computing their cohomology.  

Let $X=|K|$ be the geometric realisation of a finite simplicial complex $K$, stratified by the relative interiors of the simplices,  and $p$ be a GM-perversity. For each $\sigma\in K$ let $\imath_\sigma\colon\sigma\hookrightarrow  X$ and $\jmath_\sigma\colon\sigma^{\circ}\hookrightarrow X$ be respectively the inclusions of the simplex and its relative interior into $X$.

Following \cite{MR1453053} we say $\sigma$ is of \defn{!-type} if $p(\sigma)=p(\partial\sigma)$ and of \defn{*-type} if $p(\sigma)=p(\partial\sigma)-1$. The zero simplex is of both types. The terminology is explained by \cite[Lemma 1.1]{MR1453053} which identifies the simple objects in $\perv{p}{X}$ as
\[
\ob{S}_{\sigma}=
\begin{cases} 
{\jmath_{\sigma}}_!\jmath_\sigma^*\kk_{\sigma}[-p(\sigma)] & \text{if $\sigma$ is of !-type} \\ 
{\jmath_{\sigma}}_*\jmath_\sigma^*\kk_{\sigma}[-p(\sigma)] & \text{if $\sigma$ is of *-type}. 
\end{cases}
\]
In this situation $\perv{p}{X}$ is a faithful heart in $\constr{c}{X}$ by \cref{thm:polischuk}, indeed it is serially faithful and hence highest weight (with respect to the geometric ordering of simplices) by \cref{thm:fhw}. The standard object is $\stan{}{\sigma} ={\jmath_{\sigma}}_!\jmath_\sigma^*\kk_{\sigma}[-p(\sigma)]$ and the costandard object is $\costan{}{\sigma} ={\jmath_{\sigma}}_*\jmath_\sigma^*\kk_{\sigma}[-p(\sigma)] = \kk_\sigma[-p(\sigma)]$. These  are the corresponding simple objects when the simplex $\sigma$ is respectively of !-type and of *-type. In particular the standard and costandard objects are (shifted) constructible sheaves.

We now recall the quiver description of the perverse sheaves. The \defn{perverse dimension of a simplex $\sigma$} as introduced by MacPherson, see \cite[\S23.6]{goresky}, is 
\[
\delta(\sigma)=
\begin{cases} p^*(\sigma) &\text{if $\sigma$ is of !-type} \\ 
-p(\sigma)& \text{if $\sigma$ is of *-type}. 
 \end{cases}
\]
The perverse dimension $\delta\colon\Z_{\geq0}\to\Z$ is injective with image the interval $[p(X),-p^*(X)]$ of length $-p^*(X)-p(X) = \dim(X)$.

By \cite[\S24.4]{goresky} there is an exact equivalence $\perv{p}{X} \simeq \mod{\palg{p}{X}}$ where $\palg{p}{X}$ is the path algebra of the quiver $\pfun{p}{Q}(X)$ with 
\begin{itemize}
\item vertices the simplices $\sigma\in K$;
\item an arrow $t_{\sigma\tau} \colon \sigma \to \tau$ if $\sigma\subseteq\tau$ or $\sigma\supseteq\tau$ and $\delta(\sigma)=\delta(\tau)+1$
\end{itemize} 
modulo the ideal  generated by relations  $\sum_{\delta(\alpha)=i} t_{\sigma\alpha}t_{\alpha\tau}$ where $\delta(\sigma)= i+1$ and $\delta(\tau)=i-1$.  The quiver $\pfun{p}{Q}(X)$ is directed and ranked by negative perverse dimension. The arrow $t_{\sigma\tau}$ induces a canonical morphism  $c_{\tau\sigma} \colon P_\tau\to P_\sigma$ between the indecomposable projective modules.

\begin{proposition}\label{proj_res_simplicial}
The constant sheaf $\kk_X$ has a projective resolution, i.e. an isomorphic object in $\der{b}{\proje{\perv{p}{X}}}$, of the form
\[     
\sh{P}^* \coloneqq  \left( 0\to \bigoplus_{\pdim{\sigma} =p^*(X)} \ob{P}_\sigma \to \cdots \to \bigoplus_{\pdim{\sigma}=-p(X)} \ob{P}_\sigma \to 0 \right),
  \]
in $\constr{c}{X}$ where each $\sh{P}_\sigma$  is in degree $\pdim{\sigma}$ and the differential has entries $d_{\sigma\tau} =c_{\sigma\tau}$. (The quiver relations show this is a complex, although not that it is exact except at the last term.)
\end{proposition}
\begin{proof}
First we reduce to the case when $X$ is a single simplex. 
There is a closed embedding $\imath \colon X \hookrightarrow \Delta^N$ for any sufficiently large $N$. We extend the perversity $p$ on $X$ to one on $\Delta^N$ by assuming that all simplices of strictly greater dimension than $\dim X$ have $*$-type, in particular we assume that $N>\dim X$ so that $\Delta^N$ is of $*$-type. Note that  the induced restriction $\imath^*$ is exact, maps the projective cover of the simple $\ob{S}_\sigma$ in $\perv{p}{\Delta^N}$ to the projective cover of the corresponding simple object in $\perv{p}{X}$,  and satisfies $\imath^*\kk_{\Delta^N} = \kk_X$. Therefore it suffices to prove the result in the special case $X=\Delta^N$. 
 
Since $\Delta^N$ is of $*$-type, $\kk_{\Delta^N}=\ob{S}_{\Delta^N}[p(\Delta^N)]$ is the (shifted) simple perverse sheaf corresponding to the open stratum. This is the simple object at the unique source of the quiver $\pfun{p}{Q}(\Delta^N)$. Since the quiver is directed and ranked by negative perverse dimension the minimal projective resolution of this simple object has the form
\[
0\to \bigoplus_{\pdim{\sigma} =p^*(X)} \ob{P}_\sigma^{n_\sigma} \to \cdots \to \bigoplus_{\pdim{\sigma}=-p(X)} \ob{P}_\sigma^{n_\sigma}  \to 0
\] 
for suitable multiplicities $n_\sigma\in \N$ where the component of the differential $P_\sigma \to P_\tau$  is either  the canonical morphism $c_{\sigma\tau}$ or zero according to whether the simple $S_\sigma$ appears in the top of the kernel at that stage, or not.

In fact, by considering separately the cases when $\sigma$ is of $!$ and $*$-type one can compute
\[
\Ext{\constr{c}{X}}{i}{\kk_{\Delta^N}}{\ob{S}_\sigma} = 
\begin{cases}
\kk &  i=-\delta(\sigma)\\
0 &  \text{otherwise}. 
\end{cases}
\]
Hence $n_\sigma=1$ for all $\sigma$ as claimed. By symmetry of the simplex, if $d_{\sigma\tau}=0$ then $d_{\sigma\tau'}=0$ for all $\tau'$ with $\pdim{\tau'}=\pdim{\tau}$. This is impossible because we chose the resolution to be minimal. Hence $d_{\sigma\tau}=c_{\sigma\tau}$ for all $\sigma$ and $\tau$ as claimed.
\end{proof}

Now consider a general perverse sheaf $\sh{E}$. Let $[\sh{E}:S_\sigma]$ be  the multiplicity of the simple $S_\sigma$ in the Jordan--H\"older filtration of $\sh{E}$. Applying $\mor{}{-}{\sh{E}}$ to the above projective resolution $\sh{P}^*$ we obtain a complex
\[
0\to \bigoplus_{\pdim{\sigma} =-p(X)} k^{[\sh{E}:S_\sigma]}  \to \cdots \to \bigoplus_{\pdim{\sigma}=p^*(X)} k^{[\sh{E}:S_\sigma]}  \to 0
\]
of vector spaces indexed by the simplices with first term in degree $p(X)$. This is nothing but the quiver description of $\sh{E}$, which we can write in this way as a complex because of the form of the relations.  To see this note that the projective resolution is the projective generator in $\perv{p}{X}$, but written as a complex by shifting degrees. Morphisms from the projective generator induce an equivalence $\constr{c}{X} \simeq \der{b}{\mod{\palg{p}{X}}}$, but because of the degree shifts we obtain the quiver description of each perverse sheaf as a complex too. This description is referred to as a \defn{cellular perverse sheaf} \cite[Definition 23.11]{goresky}.
\begin{corollary}[{\cf \cite[\S 23]{goresky}}]
\label{cor:cellular perverse sheaves}
The hypercohomology $H^{*}(X;\sh{E})$ is the cohomology of the complex $\mor{}{\sh{P}^*}{\sh{E}}$.
\end{corollary}
\begin{proof}
By definition  $H^k(X;\sh{E}) = \mor{}{\kk_X}{\sh{E}[k]} = H^k( \mor{}{\sh{P}^*}{\sh{E}} )$.
\end{proof}

\begin{example}
    Let $p$ be the zero perversity, so that $\perv{p}{X}$ is the constructible sheaves on $X=|K|$ stratified by the relative interiors of the simplices. The simple constructible sheaves are $\ob{S}_\sigma = {\jmath_\sigma}_!\jmath_\sigma^*\kk_X$, the injective constructible sheaves are the $\ob{I}_\sigma={\imath_\sigma}_*\imath_\sigma^*\kk_X$, and the projective constructible sheaves are the $\ob{P}_\sigma={\jmath_{\star{\sigma}}}_!\jmath_{\star{\sigma}}^*\kk_X$ where $\jmath_{\star{\sigma}}\colon \bigcup_{\sigma \subseteq \tau} \tau^\circ \hookrightarrow X$ is the inclusion of the open star of $\sigma$.
    
    The constant sheaf $\kk_X$ is perverse and has minimal projective resolution

\[     
 0\to \bigoplus_{\dim\sigma =\dim X} \ob{P}_\sigma \to \cdots \to  \bigoplus_{\dim\sigma=1} \ob{P}_\sigma \to \bigoplus_{\dim\sigma=0} \ob{P}_\sigma \to 0.
  \] 

  Applying \cref{cor:cellular perverse sheaves} to the simple object $\ob{S}_{\sigma}={\jmath_\sigma}_!\jmath_\sigma^*\kk_X$ computes the compactly supported cohomology 
  \[
  H^r(X; {\jmath_\sigma}_!\jmath_\sigma^*\kk_{\sigma} ) = 
  \begin{cases} \kk \quad \hbox{if} \quad r= \dim\sigma \\ 
  0 \quad \hbox{otherwise}
  \end{cases}
  \]
  of the interior of the corresponding simplex. Taking morphisms into the constant sheaf $\kk_X$ instead yields the cochain complex of $X$. 
\end{example}

\begin{remark}
Suppose $X$ is a Whitney stratified space and $p$ is a GM-perversity. Goresky \cite{Goresky_1978} showed that $X$ has a triangulation $X_T$ such that the closure of each stratum is a sub-complex of $X_T$. If $\sh{E}\in \perv{p}{X}$ then $\sh{E}\in \perv{p}{X_T}$ --- the analogue is true whenever one stratification refines another. Therefore the procedure above provides a complex computing the hypercohomology of $\sh{E}$, \ie we have simplicial approach to computing the hypercohomology of any GM-perverse sheaf on any Whitney stratified space. Of course this will only be effective when there is some reasonable triangulation.
\end{remark}

\bibliographystyle{myalpha}
\bibliography{references1.bib}

\end{document}